\documentclass[oneside,reqno,11pt,a4paper]{amsart}
\usepackage[T1]{fontenc}
\usepackage[margin=2.4cm]{geometry}
\usepackage{amsmath,amsfonts,amsbsy,amssymb,amscd,amsthm}
\usepackage[usenames,dvipsnames,svgnames]{xcolor}
\usepackage[utf8]{inputenc}
\usepackage[ruled]{algorithm2e}
\usepackage{graphicx}
\usepackage{svg}
\usepackage{caption}
\usepackage{subcaption}
\usepackage{doi}
\usepackage[foot]{amsaddr}
\usepackage{makecell}
\usepackage{hyperref}
\usepackage{acronym}
\usepackage{listings}
\usepackage{textgreek}
\usepackage{url}
\usepackage{color}
\usepackage{framed}
\usepackage{verbatim}
\usepackage{fancyvrb}
\usepackage{stmaryrd}
\usepackage{newtxtext}
\usepackage{newtxmath}
\usepackage{microtype} 
\usepackage[final]{pdfpages}
\usepackage{tikz}
\usepackage{booktabs}
\usepackage{multirow}
\usepackage{amsmath}

\usepackage[backend=biber,
sorting=none,
defernumbers=false,
maxbibnames=5,
style=numeric-comp,
isbn=false,
bibencoding=utf8,
safeinputenc, 
url=false,
doi=true,
eprint=false,
giveninits=true]{biblatex}
\hypersetup{
breaklinks=true,
bookmarksopen=true,
pdftitle={Compatible Finite Element Interpolated Neural Networks},    
pdfauthor={Santiago Badia, Wei Li and Alberto F. Mart{\'{\i}}n},     
colorlinks=true,       
linkcolor=blue,          
citecolor=blue,        
filecolor=black,      
urlcolor=blue           
}  
\definecolor{bg}{rgb}{0.93,0.93,0.93}
  
\newtheorem{theorem}{Theorem}[section]

\newtheorem{proposition}[theorem]{Proposition}

\acrodef{pde}[PDE]{partial differential equation}
\acrodef{rhs}[RHS]{right-hand side}
\acrodef{fe}[FE]{finite element}
\acrodef{fem}[FEM]{finite element method}
\acrodef{fdm}[FDM]{finite difference method}
\acrodef{fvm}[FVM]{finite volume method}
\acrodef{dof}[DoF]{degree of freedom}
\acrodefplural{dof}[DoFs]{degrees of freedom}
\acrodef{nn}[NN]{neural network}
\acrodef{cnn}[CNN]{convolutional \ac{nn}}
\acrodef{pinn}[PINN]{physics-informed \ac{nn}}
\acrodef{vpinn}[VPINN]{variational \ac{pinn}}
\acrodef{ivpinn}[IVPINN]{interpolated \ac{vpinn}}
\acrodef{feinn}[FEINN]{\ac{fe} interpolated \ac{nn}}
\acrodef{picnn}[PICNN]{physics-informed \ac{cnn}}
\acrodef{gmg}[GMG]{geometric multigrid}
\acrodef{spd}[SPD]{symmetric positive definite}
\acrodef{rt}[RT]{Raviart-Thomas}
\acrodef{dg}[DG]{discontinuous Galerkin}
\acrodef{cg}[CG]{continuous Galerkin}

\graphicspath{{../plots}{./}}
\newcommand{\fig}[1]{Fig.~\ref{#1}}

\newcommand{\sect}[1]{Sect.~\ref{#1}}

\newcommand{\bdiv}{\textbf{div}}
\newcommand{\bcurl}{\textbf{curl}}
\newcommand{\norm}[1]{\left\lVert #1 \right\rVert}
\newcommand{\ltwonorm}[1]{\left\lVert #1 \right\rVert _{L^2(\Omega)}}
\newcommand{\honenorm}[1]{\left\lVert #1 \right\rVert _{H^1(\Omega)}}
\newcommand{\ltwonormd}[1]{\left\lVert #1 \right\rVert _{L^2(\Omega)^d}}
\newcommand{\hcurlnorm}[1]{\left\lVert #1 \right\rVert _{H(\bcurl; \Omega)}}
\newcommand{\hdivnorm}[1]{\left\lVert #1 \right\rVert _{H(\bdiv; \Omega)}}

\newcommand{\argmin}[1]{\underset{#1}{\mathrm{arg\,min}}\,}

\begin{document}

\title[Compatible finite element interpolated neural networks]{Compatible finite element interpolated neural networks}
\author{Santiago Badia$^1$}
\email{santiago.badia@monash.edu}
\author{Wei Li$^{1,*}$}
\email{wei.li@monash.edu}
\author{Alberto F. Mart\'{\i}n$^2$}
\email{alberto.f.martin@anu.edu.au}
\address{$^1$ School of Mathematics, Monash University, Clayton, Victoria 3800, Australia.}
\address{$^2$ School of Computing, The Australian National University, Canberra ACT 2600, Australia.}
\address{$^*$ Corresponding author.}

\date{\today}

\begin{abstract}
We extend the finite element interpolated neural network (FEINN) framework from partial differential equations (PDEs) with weak solutions in $H^1$ to PDEs with weak solutions in $H(\bcurl)$ or $H(\bdiv)$.
To this end, we consider interpolation trial spaces that satisfy the de Rham Hilbert subcomplex, providing stable and structure-preserving neural network discretisations for a wide variety of PDEs. 
This approach, coined compatible FEINNs, has been used to accurately approximate the $H(\bcurl)$ inner product. We numerically observe that the trained network outperforms  finite element solutions by several orders of magnitude for smooth analytical solutions. 
Furthermore, to showcase the versatility of the method, we demonstrate that compatible FEINNs achieve high accuracy in solving surface PDEs such as the Darcy equation on a sphere. Additionally, the framework can integrate adaptive mesh refinements to effectively solve problems with localised features. We use an adaptive training strategy to train the network on a sequence of progressively adapted meshes.
Finally, we compare compatible FEINNs with the adjoint neural network method for solving inverse problems. We consider a one-loop algorithm that trains the neural networks for unknowns and missing parameters using a loss function that includes PDE residual and data misfit terms. The algorithm is applied to identify space-varying physical parameters for the $H(\bcurl)$ model problem from partial, noisy,  or boundary observations. We find that compatible FEINNs achieve accuracy and robustness comparable to, if not exceeding, the adjoint method in these scenarios.
\end{abstract}


\maketitle

\section{Introduction} \label{sec:intro}
Conventional numerical discretisations of \acp{pde} rely on a partition of the domain (mesh) to approximate a continuous problem. The most widely used discretisation methods are the \ac{fem}, the \ac{fdm}, and the \ac{fvm}. Among these, \ac{fem} is very popular due to its flexibility to handle complex geometries, ability to achieve high-order accuracy, and suitability to tackle mixed formulations. Additionally, it is supported by a solid mathematical foundation~\cite{Ern2021}, ensuring stability and accuracy. 
Over the past few decades, optimal \ac{fem} solvers for both linear and nonlinear \acp{pde} have been developed, effectively exploiting the capabilities of large-scale supercomputers~\cite{Badia2016}.

Various physical phenomena in science and engineering can be modelled using \acp{pde}. 
As a classical example, the Poisson equation, with weak solution in $H^1$, describes phenomena such as heat conduction. The Maxwell's equations, with weak solution in $H(\bcurl)$, succinctly state the fundamentals of electricity and magnetism, while the Darcy equation, with weak flux and pressure solution in $H(\bdiv)$ and $L^2$ spaces, respectively, describes fluid flow through porous media. 
These spaces are connected through the differential operators grad, curl, and div, such that applying an operator to functions in one space maps surjectively onto the kernel of the next operator in the sequence, thereby forming a de Rham complex~\cite{Arnold2006}.

Compatible \acp{fe} are structure-preserving discretisation methods for solving these equations. They search for approximate solutions in subspaces that, at the discrete level, preserve the same relation via differential operators as their continuous, infinite-dimensional counterparts, thus forming a discrete de Rahm complex.
A common property of these complexes is the gradual reduction of inter-element continuity requirements from one space to the next: the discrete subspaces are made of piecewise polynomials that are continuous across cell boundaries in the case of $H^1$ (thus having a well defined global weak gradient), only tangential or only normal components are continuous for $H(\bcurl)$ and $H(\bdiv)$, respectively (thus having well-defined global weak curl and div, respectively), and completely discontinuous across cell boundaries for $L^2$. 
Although quite established in several application areas, compatible \acp{fe} have become increasingly popular, e.g., for the simulation of geophysical flows in the context of atmospheric and ocean modelling, mostly because their ability to effectively address the issue of spurious numerical waves present in other numerical schemes, while allowing for conservation of energy and other quantities; see~\cite{Cotter2023} for a recent survey. 

For a given polytope (e.g., a triangle or a quadrilateral) and polynomial order, there are several possible choices for the discrete subspaces that lead to the above discrete complex structure.
In this work we leverage first-kind N{\'e}d{\'e}lec~\cite{Nedelec1980} and \ac{rt}~\cite{Raviart1977} vector-valued \acp{fe} for $H(\bcurl)$ and $H(\bdiv)$ spaces, respectively (apart from the more standard continuous and discontinuous Lagrangian nodal \acp{fe} for $H^1$ and $L^2$ spaces, respectively). These \acp{fe} define the local polynomial bases and \acp{dof} in the form of integrals (moments) over mesh edges, faces, and cells so as to ensure the above mentioned continuity constraints across element boundaries. 
We note that N{\'e}d{\'e}lec \acp{fe} are particularly well-suited for Maxwell's equations as, in contrast to nodal \acp{fe}, they avoid spurious solutions, even in the case of domains with re-entrant corners or edges~\cite{Costabel2002}. They can accurately approximate discontinuous fields due to large jumps in the physical properties of the material, and are better understood mathematically than other discretisation methods for the Maxwell's equations~\cite{Monk2019}. 

Deep learning techniques, especially \acp{nn}, have gained significant popularity over the last few years for solving \acp{pde}. The main idea is that one seeks an approximation to the solution of a \ac{pde} from a trial finite-dimensional {\em nonlinear manifold} (e.g., a neural network) as opposed to a finite-dimensional linear space as in \ac{fem}. One of the most notable methods is \acp{pinn}~\cite{Raissi2019}. Instead of solving algebraic systems of equations to find an approximate solution, \acp{pinn} minimise the strong \ac{pde} residual evaluated at a set of randomly sampled collocation points. 
\Acp{pinn} have demonstrated relative success for forward (where only the solution/state is unknown) and inverse (where incomplete data, e.g., physical coefficients, is supplemented with observations of the state) \ac{pde}-constrained problems (see, e.g.,~\cite{Pang2019,Yang2021}).

The \ac{vpinn} method~\cite{Kharazmi2021} utilises a loss function based on the variational or weak form of the \ac{pde} in order to weaken the regularity requirements on the solution.  \Acp{vpinn} support $h$-refinement through domain decomposition and $p$-refinement via projection onto higher-order polynomial spaces. One of the issues in \acp{vpinn} and \acp{pinn} is the difficulties associated to the accurate integration of \acp{nn} (and their derivatives)~\cite{Magueresse2024}. To address the integration challenge in \acp{vpinn}, the \ac{ivpinn} method~\cite{Berrone2022} proposes using polynomial interpolations of \acp{nn} as trial functions in the variational formulation. \Acp{ivpinn} reduce computational costs compared to \acp{vpinn} and demonstrate significantly higher accuracy, especially when the solution is singular. Despite the improvements, these variational methods make use of the Euclidean norm of the discrete residual, which is not an appropriate measure of the error, since the residual is a functional in the dual space. As a result, the analysis for \acp{ivpinn} is sub-optimal.  

Another challenge in \ac{pinn}-based methods is the imposition of essential boundary conditions~\cite{Sukumar2022}. These methods either impose the boundary conditions weakly through a penalty term~\cite{Raissi2019} or Nitsche's method~\cite{Magueresse2024}, or rely on a combination of two functions: an offset function that satisfies the boundary conditions and a distance function that vanishes at the boundary~\cite{Berrone2022}. 

On the other hand, the so-called \ac{feinn} method, proposed in~\cite{Badia2024}, aims to find a function among all possible realisations of the nonlinear \ac{nn} manifold whose interpolation onto a trial \ac{fe} space minimises a discrete dual norm, over a suitable test \ac{fe} space, of the weak residual functional associated to the \ac{pde}.
The \ac{feinn} method imposes essential boundary conditions more naturally at the \ac{fe} space level by interpolating the \acp{nn} onto a \ac{fe} space that (approximately) satisfies the Dirichlet boundary conditions. The integration of the loss function can efficiently be handled using a Gaussian quadrature, as in \ac{fem}. 
The dual norm of the residual is an accurate measure of the error and the analysis in~\cite{Badia2024} shows that the interpolation of the \acp{nn} onto the \ac{fe} space is a stable and accurate approximation of the solution. We refer to~\cite{Rojas2024-it} for the use of the dual norm of the residual in \acp{vpinn}.

\Acp{feinn} have demonstrated exceptional performance over \ac{fem} when the target solution is smooth: the trained \acp{nn} outperform \ac{fem} solutions by several orders of magnitude in terms of $L^2$ and $H^1$ errors, even with complex geometries. With minimal modifications, \acp{feinn} can also be extended to solve inverse problems and exhibit comparable performance to the adjoint-based \ac{nn} approach~\cite{Badia2024}.
Besides, when combined with adaptive meshing techniques, $h$-adaptive \acp{feinn}~\cite{Badia2024adaptive} can efficiently solve \acp{pde} featuring sharp gradients and singularities while unlocking the nonlinear approximation power of discrete \ac{nn} manifold spaces. The trained \acp{nn} show potential to achieve higher accuracy than \ac{fem}, particularly when the solution is not singular~\cite{Badia2024adaptive}.

Most existing \ac{nn} discretisation \ac{pde} solvers are developed for problems with weak solutions in $H^1$, such as the Poisson equation~\cite{Magueresse2024,Badia2024adaptive}, or in $H^1 \times L^2$, such as the Navier-Stokes equations~\cite{Cai2021,Cheng2021,Pichi2023}. However, there are limited studies attacking problems with weak solutions in $H(\bcurl)$ or $H(\bdiv)$.
In~\cite{Baldan2021}, the authors employ \acp{nn} to solve 1D nonlinear magneto quasi-static Maxwell's equations in frequency or time domains.
Similar to \acp{pinn}, MaxwellNet~\cite{Lim2022} trains a \ac{cnn} using the strong \ac{pde} residual of the Maxwell's equations as the loss function. The authors in~\cite{Baldan2023} address inverse Maxwell's problems by integrating a hypernetwork into the \ac{pinn} framework. This addition enables the trained hypernetwork to act as a parametrised real-time field solver, allowing rapid solutions to inverse electromagnetic problems.

There are even fewer works on \ac{nn}-based solvers for the Darcy equation. 
The \ac{picnn} method proposed in~\cite{Zhang2023} uses a \ac{cnn} to simulate transient two-phase Darcy flows in homogeneous and heterogeneous reservoirs. To ensure flux continuity, they adopt a \ac{fvm} to approximate the PDE residual in the loss function. 
In~\cite{He2020}, the authors employ multiple \acp{nn} to approximate both the unknown parameters and states of the inverse Darcy systems.  Their study shows that \acp{pinn} offer regularisation and decrease the uncertainty in \ac{nn} predictions.

Another interesting topic is the development of \ac{nn} solvers for \acp{pde} posed over immersed manifolds, e.g., the Darcy equation defined on a sphere. There are a few works on \acp{pinn} for surface \acp{pde} in the literature.
In~\cite{Fang2020}, the authors extend \acp{pinn} to solve the Laplace-Beltrami equation on 3D surfaces.
The authors in~\cite{Hu2024} utilise a single-hidden-layer \acp{pinn} to solve the Laplace-Beltrami and time-dependent diffusion equations on static surfaces, as well as advection-diffusion equations on evolving surfaces. 
Another related work is~\cite{Bihlo2022}, where the authors apply \acp{pinn} to solve the shallow-water equations on the sphere.

In this work, we integrate compatible \acp{fe}, i.e., spaces that form a discrete the Rham complex~\cite{Arnold2006}, into the \ac{feinn} method proposed in~\cite{Badia2024}. This integration enables us to solve \acp{pde} with weak solutions in $H(\bcurl)$ or $H(\bdiv)$. We refer to this method as compatible \acp{feinn}. 
The novelty of compatible \acp{feinn} over the standard \ac{feinn} method in~\cite{Badia2024} lies in introducing curl/div-conforming trial \ac{fe} spaces to interpolate \acp{nn}, ensuring the desired structure-preserving properties across element boundaries. The residual minimisation framework underlying \acp{feinn} allows for some flexibility in selecting the trial and test \ac{fe} spaces provided that an inf-sup compatibility condition is satisfied among them; see~\cite{Badia2024adaptive} for the numerical analysis of the method. 
As we showed in~\cite{Badia2024} and further confirmed in this paper for the problems at hand, a proper choice of these spaces can lead to significant improvements in the accuracy of the solution and the convergence of the optimiser. 
As an evidence on the soundness of this approach, we also showcase its applicability to solve \acp{pde} posed over immersed manifolds. We interpolate \ac{nn} vector-valued fields living in ambient space (and thus not necessarily tangent to the manifold) onto \ac{fe} spaces made of functions constrained to be tangent to the manifold by construction. Strikingly, \acp{nn} trained  on a coarse mesh and lower-order bases are able to provide several orders of magnitude higher accurate solutions to the problem when interpolated onto a \ac{fe} space built out of finer meshes and higher order bases. Besides, this method, which resembles TraceFEM~\cite{Olshanskii2017} (as it seeks an approximation of the surface \ac{pde} on a finite-dimensional space of functions living on a higher dimensional space), does not need stabilisation terms in the loss function to enforce tangentiality, as it relies on a surface \ac{fe} interpolation on the tangent space. A comprehensive set of numerical experiments is conducted to demonstrate the performance of compatible \acp{feinn} in forward and inverse problems involving \acp{pde} in $H(\bcurl)$ and $H(\bdiv)$.

The rest of the paper is organised as follows. 
In \sect{sec:method}, we introduce model problems for the Maxwell's equations and Darcy equation considered in this article, their compatible \ac{fe} discretisation, and the compatible \ac{feinn} method in both forward and inverse scenarios.
\sect{sec:experiments} starts with a brief discussion on the implementation, and then presents the numerical results of the forward and inverse experiments.
Finally, we conclude the paper in \sect{sec:conclusions}.

\section{Methodology} \label{sec:method}

\subsection{Notation} \label{sec:method:notation}
In this section we introduce some essential mathematical notation that will be required for the rest of the paper. Let $\Omega \subset \mathbb{R}^d$ be a Lipschitz polyhedral domain, where $d \in \{2, 3\}$ is the dimension of the space where $\Omega$ lives. We denote the boundary of $\Omega$ by $\partial \Omega$, the Dirichlet boundary by $\Gamma_D \subset \partial \Omega$, and the Neumann boundary by $\Gamma_N \subset \partial \Omega$. Note that $\Gamma_D \cup \Gamma_N = \partial \Omega$ and $\Gamma_D \cap \Gamma_N = \emptyset$.
We consider a scalar-valued function $p: \Omega \rightarrow \mathbb{R}$, a vector-valued function $\mathbf{u}: \Omega \rightarrow \mathbb{R}^d$, and the standard square-integrable spaces $L^2(\Omega)$ for scalar functions and $L^2(\Omega)^d$ for vector functions. The corresponding norms are denoted by $\ltwonorm{\cdot}$ and $\ltwonormd{\cdot}$, respectively.

In three dimensions ($d=3$), we have the following spaces:
\begin{align*}
    H^1(\Omega) &= \{ p \in L^2(\Omega): \pmb{\nabla} p \in L^2(\Omega)^d \}, \\
    H(\bcurl; \Omega) &= \{ \mathbf{u} \in L^2(\Omega)^d: \pmb{\nabla} \times \mathbf{u} \in L^2(\Omega)^d \}, \\
    H(\bdiv; \Omega) &= \{ \mathbf{u} \in L^2(\Omega)^d: \pmb{\nabla} \cdot \mathbf{u} \in L^2(\Omega) \},
\end{align*}
where $\pmb{\nabla}$, $\pmb{\nabla} \times$, $\pmb{\nabla} \cdot$ denote the weak gradient, curl, and divergence differential operators, respectively. Detailed definitions of these spaces and operators can be found in, e.g.,~\cite{Ern2021}. We equip these spaces with the following norms:
\begin{align*}
    \honenorm{p} &= \left(\ltwonorm{p}^2 + \ltwonormd{\pmb{\nabla} p}^2\right)^{1/2}, \\
    \hcurlnorm{\mathbf{u}} &= \left( \ltwonormd{\mathbf{u}}^2 + \ltwonormd{\pmb{\nabla} \times \mathbf{u}}^2 \right)^{1/2}, \\
    \hdivnorm{\mathbf{u}} &= \left( \ltwonormd{\mathbf{u}}^2 + \ltwonorm{\pmb{\nabla} \cdot \mathbf{u}}^2 \right)^{1/2}.
\end{align*}

In two dimensions ($d=2$), the curl operator is defined as the divergence of a 90 degrees counter-clockwise rotation of the input vector, and thus produces scalar functions (as opposed to vector-valued functions for $d=3$). Thus, for the curl case, we have for $d=2$:
\begin{equation*}
    H(\bcurl; \Omega) = \{ \mathbf{u} \in L^2(\Omega)^2: \pmb{\nabla} \times \mathbf{u} \in L^2(\Omega) \}, \quad \hcurlnorm{\mathbf{u}} = \left( \ltwonormd{\mathbf{u}} + \ltwonorm{\pmb{\nabla} \times \mathbf{u}} \right)^{1/2}.
\end{equation*}

\subsection{Continuous problems} \label{sec:method:continuous}
As model problems, we consider two types of \acp{pde}: the inner product in $H(\bcurl)$ on a $d$-dimensional domain (as a simplified model problem for the Maxwell's equations) and the Darcy equation defined on a 2-dimensional closed surface (i.e., manifold) embedded in 3-dimensional space.
In the following, we denote the $H(\bcurl)$ problem as the Maxwell problem for brevity, even though it is a simplification of the full Maxwell's system. 

\subsubsection{Maxwell's equations} \label{sec:method:continuous:maxwell}

As a model problem for the Maxwell's equations, we consider the $H(\bcurl)$-inner product problem: find the field $\mathbf{u}: \Omega \rightarrow \mathbb{R}^d$ such that
\begin{equation} \label{eq:maxwell_strong}
    \pmb{\nabla} \times \pmb{\nabla} \times \mathbf{u} + \kappa \mathbf{u} = \mathbf{f} \text{ in } \Omega, \quad \mathbf{u} \times \mathbf{n} = \mathbf{g} \text{ on } \Gamma_D = \partial \Omega,
\end{equation}
where $\kappa$ is a scalar-valued material parameter field, $\mathbf{f}$ is a source term, $\mathbf{n}$ is the outward unit normal vector to $\partial \Omega$, and $\mathbf{g}$ is the Dirichlet data. Note that, for simplicity and without loss of generality, we consider pure Dirichlet boundary conditions (i.e., $\Gamma_N = \emptyset$). Imposition of Neumann boundary conditions is straightforward and can be included in the formulation by a simple modification of the \ac{rhs}.

Let $U^1 \doteq H(\bcurl; \Omega)$ and its subspace $U^1_0 \doteq \{ \mathbf{v} \in H(\bcurl; \Omega): \mathbf{v} \times \mathbf{n} = 0 \text{ on } \partial \Omega \}$. Consider the bilinear and linear forms: 
\begin{equation*}
    a(\mathbf{u}, \mathbf{v}) = \int_{\Omega} (\pmb{\nabla} \times \mathbf{u}) \cdot (\pmb{\nabla} \times \mathbf{v}) + \kappa \mathbf{u} \cdot \mathbf{v}, \quad \ell(\mathbf{v}) = \int_{\Omega} \mathbf{f} \cdot \mathbf{v}.
\end{equation*}
The weak form of~\eqref{eq:maxwell_strong} reads: find the solution $\mathbf{u} + \bar{\mathbf{u}} \in U^1$, with $\mathbf{u} \in U^1_0$ such that
\begin{equation} \label{eq:maxwell_weak}
    a(\mathbf{u}, \mathbf{v}) = \ell(\mathbf{v}) - a(\bar{\mathbf{u}}, \mathbf{v}), \quad \forall \mathbf{v} \in U_0^1,
\end{equation}
and $\bar{\mathbf{u}}$ an offset (a.k.a. lifting) function satisfying the Dirichlet boundary conditions. The well-posedness of this problem can be proven invoking the 
Lax-Milgram lemma~\cite{Ern2021b}. We also define the weak \ac{pde} residual functional for the problem as
\begin{equation}
    \mathcal{R}(\mathbf{u}) = a(\mathbf{u} + \bar{\mathbf{u}}, \cdot) - \ell(\cdot) \in {U_0^1}',
\end{equation} 
where ${U_0^1}'$ denotes the dual space of $U_0^1$.

\subsubsection{Darcy equation on a closed surface} \label{sec:method:continuous:darcy}

The strong form of the Darcy equations defined on a closed two-dimensional manifold $S$ embedded in three dimensions read: find the flux $\mathbf{u}: S \rightarrow \mathbb{R}^3$ and the pressure $p: S \rightarrow \mathbb{R}$ such that
\begin{equation} \label{eq:darcy_strong}
    \pmb{\nabla}_S \cdot \mathbf{u} = f \text{ in } S, \quad \mathbf{u} = -\pmb{\nabla}_S p \text{ in } S,
\end{equation}
where $\pmb{\nabla}_S \cdot$ and $\pmb{\nabla}_S$ denote the divergence and gradient operators on the manifold $S$, respectively. The vector-valued fields $\mathbf{u}$ and $\pmb{\nabla}_S p$ are, by definition, constrained to be tangent to $S$. Note that boundary conditions are not part of the system due to the closed nature of $S$.

We define the spaces $U^2 \doteq H(\bdiv; S)$, $U^3 \doteq L^2(S)$ and $\tilde{U}^3 \doteq L^2_0(S) \doteq \{ q \in L^2(S): \int_{S} q = 0\}$. We denote $(\cdot, \cdot)$ as the $L^2(S)$ inner product, i.e., $(p, q) = \int_{S} p q$.
The weak form of~\eqref{eq:darcy_strong} reads: find $(\mathbf{u}, p) \in U^2 \times \tilde{U}^3$ such that
\begin{equation} \label{eq:darcy_weak}
    (\mathbf{u}, \mathbf{v}) - (p, \pmb{\nabla}_S \cdot \mathbf{v}) = 0 \quad \forall \mathbf{v} \in U^2, \qquad
    (\pmb{\nabla}_S \cdot \mathbf{u}, q) = (f, q) \quad \forall q \in \tilde{U}^3,
\end{equation}
or, equivalently, the mixed formulation:
\begin{equation} \label{eq:darcy_weak2}
    \mathcal{A}((\mathbf{u}, \mathbf{v}), (p, q)) = \mathcal{L}((\mathbf{v}, q)) \quad \forall (\mathbf{v}, q) \in U^2 \times \tilde{U}^3,
\end{equation}
with the mixed forms defined as
\begin{equation*}
    \mathcal{A}((\mathbf{u}, \mathbf{v}), (p, q)) \doteq (\mathbf{u}, \mathbf{v}) - (p, \pmb{\nabla}_S \cdot \mathbf{v}) + (\pmb{\nabla}_S \cdot \mathbf{u}, q), \quad \mathcal{L}((\mathbf{v}, q)) \doteq (f, q).    
\end{equation*}
We can also define the weak \ac{pde} residual for this problem as
\begin{equation}
    \mathcal{R}(\mathbf{u}, p) = \mathcal{A}((\mathbf{u}, p), (\cdot,\cdot)) - \mathcal{L}((\cdot,\cdot)) \in {U^2}' \times \tilde{U}^{3'},
\end{equation}
where ${U^2}'$ and $\tilde{U}^{3'}$ are the dual spaces of $U^2$ and $\tilde{U}^3$, respectively. The well-posedness of this problem can be proven invoking the Babu\v{s}ka-Brezzi theory~\cite{Ern2021b}.

\subsection{Finite element discretisation} \label{sec:method:compatible}

\subsubsection{Discrete de Rham complexes} \label{sec:method:comfeinn:rham}
Compatible \ac{fe} spaces form a discrete differential complex, namely a de Rham complex~\cite{Ern2021,Arnold2006}. These complexes comprise a sequence of spaces related by differential operators. 
For $\Omega \subset \mathbb{R}^3$, it reads as follows:
\begin{equation} \label{eq:de_rham3d}
    \begin{CD}
      U^0 = H^1(\Omega) @> \text{d}^1 = \pmb{\nabla}  >>
      U^1 = H(\bcurl; \Omega) @> \text{d}^2 = \pmb{\nabla} \times >>
      U^2 = H(\bdiv; \Omega) @> \text{d}^3 = \pmb{\nabla} \cdot >> 
      U^3 = L^2(\Omega) \\
      @VV{\pi_h^0}V @VV{\pi_h^1}V @VV{\pi_h^2}V @VV{\pi_h^3}V \\
      U_h^0 @> \text{d}^1 = \pmb{\nabla} >> 
      U_h^1 @> \text{d}^2 = \pmb{\nabla} \times >> 
      {U_h^2} @> \text{d}^3= \pmb{\nabla} \cdot >> 
      {U_h^3}. \\
    \end{CD}
\end{equation}
In the top row, we have the continuous spaces, and in the bottom row, we have the corresponding discrete subspaces, i.e., $U_h^i \subset U^i$. 
The differential operators $\text{d}^i$ map between the continuous spaces in such a way that the image  of $\text{d}^i$ is the kernel of $\text{d}^{i+1}$, i.e., $\text{d}^{i+1} \circ \text{d}^{i} = 0$ for $i \in \{1, 2\}$.
This very same relation is preserved for the discrete spaces in the bottom row, forming a discrete de Rham complex.
Besides, using suitable interpolation operators $\pi_h^i$, $i \in \{0, 1, 2, 3\}$, both complexes are connected by commutative relations $\text{d}^{i}\pi_h^{i-1} u = \pi_h^{i} \text{d}^{i} u, i \in \{1, 2, 3\}$.
For $d=2$, rotating any vector field in $H(\bdiv; \Omega)$ by $\pi/2$ radians counter clock-wise results in a vector field in $H(\bcurl; \Omega)$. Therefore, the de Rham complex in two dimensions can be simplified by removing $U^1$, $\pi_h^1$, $U_h^1$, and $\text{d}^2$ from ~\eqref{eq:de_rham3d} and applying a 90-degree counter-clockwise rotation to the output of the gradient operator. 

\subsubsection{Edge and face finite elements} \label{sec:method:comfeinn:edgeface}

In order to build $U_h^1$ and ${U_h^2}$ (see~\eqref{eq:de_rham3d}) we leverage the N{\'e}d{\'e}lec edge elements of the first kind~\cite{Nedelec1980} and the \ac{rt} face elements~\cite{Raviart1977}, respectively,  out of the different options available. These \acp{fe} are well-established in the literature. We thus refer the reader to, e.g.,~\cite{Olm2019,Ern2021} for a detailed definition of their building blocks. Let $\mathcal{T}_h$ be a shape-regular partition of $\Omega$ with element size $h>0$. Let $K$ be an arbitrary cell of $\mathcal{T}_h$. For $k\geq 1$, we denote by $\mathbf{RT}_k(K)$ and $\mathbf{ND}_{k}(K)$ the \ac{rt} and N{\'e}d{\'e}lec elements of the first kind of order $k$, respectively, on cell $K$. We have that:
\begin{align*}
  U_h^1 &:= \{\mathbf{u}_h\in H(\bcurl; \Omega): \mathbf{u}_h|_K \in \mathbf{ND}_{k}(K),\quad \forall K\in \mathcal{T}_h\}, \\
  {U_h^2} &:= \{\mathbf{x}_h \in H(\bdiv; \Omega): \mathbf{x}_h|_K \in \mathbf{RT}_k(K),\quad \forall K\in \mathcal{T}_h\},
\end{align*}
with ${\pi_h^1}$ and ${\pi_h^2}$ in~\eqref{eq:de_rham3d} being the global N{\'e}d{\'e}lec and \ac{rt} interpolators~\cite{Ern2021}. For example, for $k=1$ (i.e., lowest order), these interpolation operators involve the evaluation of \acp{dof} defined as integrals over global mesh faces and edges of the normal and tangential components, respectively, of the vector field to be interpolated into the global \ac{fe} space. 

\subsubsection{Linearised test spaces} \label{sec:method:comfeinn:lin}

Following the observations in \cite{Badia2024adaptive} for grad-conforming problems, we consider a Petrov-Galerkin discretisation method for the \ac{fe} discretisation of the curl and div-conforming problems being considered in this work.
We define $V_h^i$ as the \emph{linearised}  test space corresponding to the trial space $U_{h}^i$. In the following, we discuss the construction of this linearised test space $V_h^1$ for the Maxwell problem and $V_h^2$ for the Darcy problem. 

Let $k_U$ denote the order of $U_h^i$, and let $\mathcal{T}_{h/k_U}$ represent the mesh obtained by $k_U-1$ uniform refinements of $\mathcal{T}_h$.
The test space ${V_h^i}$ is constructed out of the lowest order elements on the refined mesh $\mathcal{T}_{h/k_U}$. For example, $V_h^1$ and $V_h^2$ are built out of the lowest order N{\'e}d{\'e}lec and \ac{rt} elements, respectively, on $\mathcal{T}_{h/k_U}$. For this reason, we refer to ${V_h^i}$ as a \emph{linearised} test \ac{fe} space. In the following propositions, we show that the dimensions of $U_h^i$ and ${V_h^i}$ are equal for $i\in\{1,2\}$ in 2D and 3D, but the proof can readily be extended to any dimension and other compatible spaces~\cite{Arnold2019-hi}. The grad-conforming linearised space was studied in~\cite{Badia2024}. The result straightforwardly holds for $U_h^3$.  

\begin{proposition} \label{prop:nedelec-dim}
  The trial space $U_h^1$ and the linearised test space ${V_h^1}$ for quadrilateral and hexahedral N{\'e}d{\'e}lec elements of the first kind have the same number of \acp{dof} per geometrical entity (edge, face, cell) on the mesh $\mathcal{T}_h$ and as a result the same dimensions.
\end{proposition}

\begin{proof}
  Each quadrilateral N{\'e}d{\'e}lec element of order $k_U$ has a total of $2 k_U(k_U + 1)$ \acp{dof}: $4 k_U$ on the boundary edges (with $k_U$ \acp{dof} per edge) and $2 k_U (k_U - 1)$ in the interior.
  After applying $k_U - 1$ uniform refinements, the refined element contains $4 k_U$ boundary subedges and $2 k_U (k_U - 1)$ interior subedges. 
  Since the lowest order N{\'e}d{\'e}lec element has one \ac{dof} per edge, the linearised element has $4 k_U$ boundary \acp{dof} and $2 k_U (k_U - 1)$ interior \acp{dof}.
  Therefore, the \acp{dof} per geometrical entity and dimensions of $U_h^1$ and ${V_h^1}$ are equal for quadrilateral N{\'e}d{\'e}lec elements.

  Each hexahedral N{\'e}d{\'e}lec element of order $k_U$ has a total of $3 k_U (k_U + 1)^2$ \acp{dof}: $12 k_U$ on the edges (with $k_U$ \acp{dof} per edge), $6(2 k_U^2 - 2k_U)$ on the faces (with $2 k_U^2 - 2k_U$ \acp{dof} per face), and $3 k_U (k_U - 1)^2$ in the interior. 
  After applying $k_U - 1$ uniform refinements, the refined element has $k_U$ subedges on each of the original edges, $2 k_U^2 - 2 k_U$ subedges on each of the original faces, and $3 k_U (k_U - 1)^2$ interior subedges. 
  The linearised N{\'e}d{\'e}lec element also has $12 k_U$ edge \acp{dof}, $6(2 k_U^2 - 2k_U)$ face \acp{dof}, and $3 k_U (k_U - 1)^2$ interior \acp{dof}. 
  Thus, the \acp{dof} per geometrical entity and dimensions of $U_h^1$ and ${V_h^1}$ are also equal for hexahedral N{\'e}d{\'e}lec elements.
\end{proof}

\begin{proposition} \label{prop:rt-dim}
  The trial space ${U_h^2}$ and the linearised test space ${V_h^2}$ for quadrilateral and hexahedral \ac{rt} elements have the same number of \ac{dof}  per geometrical entity (edge, face, cell) on the mesh $\mathcal{T}_h$ and as a result the same dimensions.
\end{proposition}

\begin{proof}
  Since quadrilateral N{\'e}d{\'e}lec elements can be built by rotating \ac{rt} elements~\cite{Ern2021}, Prop.~\ref{prop:nedelec-dim} applies to quadrilateral \ac{rt} elements as well. 
  Thus,  the \acp{dof} per geometrical entity and the dimensions of ${U_h^2}$ and ${V_h^2}$ are equal for these 2D elements.

  Each hexahedral \ac{rt} element of order $k_U$ has a total of $3 k_U^2 (k_U + 1)$ \acp{dof}: $6 k_U^2$ on the faces (with $k_U^2$ \acp{dof} per face) and $3 k_U^2 (k_U - 1)$ in the interior. 
  After applying $k_U - 1$ uniform refinements, the refined element has $k_U^2$ subfaces on each of the original faces and $3 k_U^2 (k_U - 1)$ interior subfaces. 
  Since the lowest order \ac{rt} element has one \ac{dof} per face, the linearised \ac{rt} element also has $6 k_U^2$ face \acp{dof} and $3 k_U^2 (k_U - 1)$ interior \acp{dof}. 
  Thus,  the \acp{dof} per geometrical entity and the dimensions of ${U_h^2}$ and ${V_h^2}$ are also equal for hexahedral \ac{rt} elements.
\end{proof}

\subsubsection{Maxwell problem finite element discretisation} \label{sec:method:compfe}

Let us consider an offset $\bar{\mathbf{u}}_h \in U_h^1$ such that $\bar{\mathbf{u}}_h = \pi_h^1(\mathbf{g})$ on $\Gamma_D$. A common choice is to extend $\bar{\mathbf{u}}_h$ by zero on the interior.
We use the subscript $0$ to denote the continuous and discrete spaces with zero trace on $\partial \Omega$ and $\pi_{h,0}^i$ to represent the interpolation onto these subspaces. For example, $U_{h,0}^1$ is the discrete subspace of $U_h^1$ with zero tangential trace on $\partial \Omega$ and $\pi_{h,0}^1 : U_0^1 \rightarrow U_{h,0}^1$ 
With these ingredients, we can formulate the \ac{fe} problem for the Maxwell's problem as: find $\mathbf{u}_h \in U_{h,0}^1$ such that
\begin{equation} \label{eq:maxwell_fe}
    a(\mathbf{u}_h, \mathbf{v}_h) = \ell(\mathbf{v}_h) - a(\bar{\mathbf{u}}_h, \mathbf{v}_h), \quad \forall \mathbf{v}_h \in {V_{h,0}^1}.
\end{equation}
The \ac{fe} residual for this problem is defined as
\begin{equation} \label{eq:maxwell_fe_res}
    \mathcal{R}_h(\mathbf{u}_h) = a(\mathbf{u}_h + \bar{\mathbf{u}}_h, \cdot) - \ell(\cdot) \in {V_{h,0}^1}'.
\end{equation}
The well-posedness of the discrete problem depends on a inf-sup compatibility condition between the discrete spaces $U_{h,0}^1$ and ${V_{h,0}^1}$.

\subsubsection{Surface Darcy problem finite element discretisation} \label{sec:method:tracefe}

In this section, we consider the \ac{fe} discretisation of the Darcy problem on a closed surface; the restriction to a flat space is straightforward. It relies on a discrete approximation of the manifold $S$, denoted by $S_h$. 
For the particular case $S$ is a sphere, a common choice for the discrete manifold $S_h$, of particular relevance for computational 
atmospheric flow dynamics, is the so-called cubed-sphere mesh~\cite{Ronchi1996}. 
In this mesh, $S_h$ is made of quadrilateral elements $K\subset \mathbb{R}^3$ (thus living in 3D ambient space) that are the image under a diffeomorphism $\Phi_K$ of a reference square $\hat{K} \subset \mathbb{R}^2$. That is, we have that $K=\Phi_K(\hat{K})$. We denote by $J_K$ the Jacobian of $\Phi_K$, and stress that $J_K$ has three rows and two columns (as $\Phi_K$ has three components, and depends on two parameters). We use $G_K=J_K^T J_K$ to denote the corresponding first fundamental form.

We pursue the approach in~\cite{Rognes2013} to build the \ac{fe} space ${U_h^2}$ on $S_h$.
This space is made of piece-wise polynomial vector-valued fields tangent to $S_h$ with normal component continuous across element interfaces. To this end, this approach relies on a $\mathbf{RT}_k(\hat{K})$ element defined on parametric space $\hat{K}$. 
Let $\mathbf{u}_h$  be an arbitrary element of ${U_h^2}$, and $\mathbf{u}_h|_K$ its restriction to $K$. Then the surface contravariant Piola mapping is applied to build $\mathbf{u}_h|_K$ out of its pullback vector-valued function $\hat{\mathbf{u}} \in \mathbf{RT}_k(\hat{K})$ in parametric space. 
Namely, we have that $\mathbf{u}_h|_K\circ \Phi_K(\hat{\mathbf{x}})=\frac{1}{\sqrt{|G_K|}}J_K\hat{\mathbf{u}}(\hat{\mathbf{x}})$.
We stress that $\hat{\mathbf{u}} \in \mathbb{R}^2$ while $\mathbf{u}_h|_K \in \mathbb{R}^3$.
A similar (but different) construction is used to compute the discrete surface divergence operator $\nabla_{S_h} \cdot \mathbf{u}_h$ on each cell. We refer to~\cite{Rognes2013} for further details.

As for the curl-conforming problem, we can also consider a Petrov-Galerkin discretisation of this problem using the linearised space $V_h^i$ for $i\in \{2,3\}$  discussed above. With these ingredients, the Petrov-Galerkin discretisation of~\eqref{eq:darcy_strong} reads as: 
find $(\mathbf{u}_h, p_h) \in {U_h^2} \times {\tilde{U}_h^3}$ such that 
\begin{align} \label{eq:darcy_fe}
    \mathcal{A}_h((\mathbf{u}_h, \mathbf{v}_h), (p_h, q_h)) = \mathcal{L}_h((\mathbf{v}_h, q_h)) \quad \forall (\mathbf{v}_h, q_h) \in {V_h^2} \times {V_h^3},
\end{align}
where $\mathcal{A}_h$ and $\mathcal{L}_h$ are the discrete mixed forms corresponding to $\mathcal{A}$ and $\mathcal{L}$, respectively, and $S$ is replaced by $S_h$, and $\pmb{\nabla}_S \cdot$ is replaced by $\pmb{\nabla}_{S_h} \cdot$. We note that this problem, as-is, does not have a unique solution. 
Another solution can be obtained by adding an arbitrary constant to the discrete pressure field $p_h$. To fix the constant, we further impose the constraint $\int_{S_h} p_h = 0$ using a Lagrange multiplier. This detail is omitted from the presentation for brevity. 
Besides, the \ac{fe} residual for this problem is defined as
\begin{equation} \label{eq:darcy_fe_res}
    \mathcal{R}_h(\mathbf{u}_h, p_h) = \mathcal{A}((\mathbf{u}_h, p_h), \cdot) - \mathcal{L}(\cdot) \in {V_h^2}' \times {V_h^3}'. 
\end{equation} 
The well-posedness of this saddle-point problem relies on the Babuska-Brezzi theory (see~\cite{Ern2021b}).

\subsection{Neural networks} \label{sec:method:nn}
We utilise fully-connected, feed-forward \acp{nn} composed of affine maps and nonlinear activation functions. We use a tuple $(n_0, \ldots n_L)\in \mathbb{N}^{(L+1)}$ to represent the network architecture, where $L$ is the number of layers, and $n_k$ (for $1 \leq k \leq L$) is the number of neurons at each layer. Clearly, $n_0 = d$; for scalar-valued \acp{nn}, $n_L = 1$; for vector-valued \acp{nn}, $n_L = d$. Following~\cite{Badia2024,Badia2024adaptive}, we adopt $n_1 = n_2 = ... = n_{L-1} = n$, meaning all the hidden layers have an equal number of neurons $n$.

At each layer $k$ (for $1 \leq k \leq L$), we denote the affine map as $\pmb{\Theta}_k: \mathbb{R}^{n_{k-1}} \to \mathbb{R}^{n_k}$, defined by $\pmb{\Theta}_k \pmb{x} = \pmb{W}_k \pmb{x} + \pmb{b}_k$, where $\pmb{W}_k \in \mathbb{R}^{n_k \times n_{k-1}}$ is the weight matrix and $\pmb{b}_k \in \mathbb{R}^{n_k}$ is the bias vector.
After each affine map except the final one, we apply an activation function $\rho: \mathbb{R} \to \mathbb{R}$ element-wise. We apply the same fixed activation function across all layers, although each layer theoretically could have a distinct (trainable) activation function.
With these definitions, the \ac{nn} is a parametrisable function $\mathcal{N}(\pmb{\theta}): \mathbb{R}^d \to \mathbb{R}^{n_L}$ defined as 
\begin{equation} \label{eq:nn_structure}
  \mathcal{N}(\pmb{\theta}) = \pmb{\Theta}_L \circ \rho \circ \pmb{\Theta}_{L-1} \circ \ldots \circ \rho \circ \pmb{\Theta}_1,
\end{equation}
where $\pmb{\theta}$ represents all the trainable network parameters $\pmb{W}_k$ and $\pmb{b}_k$. We denote the network nonlinear manifold space with $\mathcal{N}$ and a realisation of it with $\mathcal{N}(\pmb{\theta})$.

\subsection{Compatible finite element interpolated neural networks} \label{sec:method:comfeinn}
For simplicity, we focus on the Maxwell's problem~\eqref{eq:maxwell_strong} in this section. We shall discuss the Darcy problem in \sect{sec:method:tracefeinn}. Our approach to the neural discretisation relies on finding a \ac{nn} such that its interpolation into $U_{h,0}^1$ minimises a discrete dual norm of the residual $\mathcal{R}_h$ over the test space ${V_{h,0}^1}$. 

The compatible \ac{feinn} method combines the compatible \ac{fe} discretisation~\eqref{eq:maxwell_fe} and the \ac{nn}~\eqref{eq:nn_structure} by integrating a \ac{nn} $\mathbf{u}_\mathcal{N}$ into the \ac{fe} residual~\eqref{eq:maxwell_fe_res}. This leads to the following non-convex optimisation problem:
\begin{equation} \label{eq:maxwell_optim}
    \mathbf{u}_\mathcal{N} \in \argmin{\mathbf{w}_\mathcal{N} \in \mathcal{N}} \mathscr{L}(\mathbf{w}_\mathcal{N}), \qquad \mathscr{L}(\mathbf{w}_\mathcal{N}) \doteq \norm{\mathcal{R}_h(\pi^1_{h,0}(\mathbf{w}_\mathcal{N}))}_{{V^1_0}'}.
\end{equation}
Other choices of the residual norm are discussed in \sect{sec:method:loss}. {We note that some \ac{nn} architectures are not compact and the minimum cannot be attained but approximated up to any $\epsilon > 0$ in a meaningful norm. We refer to~\cite{Badia2024adaptive} for more details.} 

Compatible \acp{feinn} make use of the appropriate interpolators for the functional space in which the problem is posed, determined by the discrete de Rham complex. This approach maintains all the advantages of the \acp{feinn} in~\cite{Badia2024,Badia2024adaptive} for grad-conforming approximations. 
We stress the fact that the boundary conditions are incorporated in the method by the interpolation operator $\pi_{h,0}^1$ on the \ac{fe} space with zero trace and 
the \ac{fe} offset $\bar{\mathbf{u}}_h$. Unlike \acp{pinn}~\cite{Raissi2019,Kharazmi2021}, there is no need to incorporate these conditions as penalties in the loss, nor to impose them at the \ac{nn} level using distance functions as in~\cite{Sukumar2022, Berrone2022}.
Furthermore, since compatible \ac{fe} bases are polynomials, the integrals in the \ac{fe} residual can be computed exactly using Gaussian quadrature rules, avoiding issues with Monte Carlo integration~\cite{Rivera2022} and the need for adaptive quadratures~\cite{Magueresse2024}. Besides, relying on a discrete \ac{fe} test space, one can use the dual norm of the residual in the loss function, which is essential in the numerical analysis of the method~\cite{Badia2024adaptive} and has been experimentally observed beneficial in the training process.

\subsection{Trace finite element interpolated neural networks} \label{sec:method:tracefeinn}
Based on the approach to \ac{fe} discretisation discussed in \sect{sec:method:tracefe}, we introduce the trace \ac{feinn} method for approximating~\eqref{eq:darcy_strong}. The optimisation problem involves two \acp{nn} $\mathbf{u}_\mathcal{N}$ and $p_\mathcal{N}$, and reads: 
\begin{equation} \label{eq:surface_darcy_optim}
    \mathbf{u}_\mathcal{N}, p_\mathcal{N} \in \argmin{\mathbf{w}_\mathcal{N}, q_\mathcal{N} \in \mathcal{N}_\mathbf{u} \times \mathcal{N}_p} \mathscr{L}(\mathbf{w}_\mathcal{N}, q_\mathcal{N}), \quad \mathscr{L}(\mathbf{w}_\mathcal{N}, q_\mathcal{N}) \doteq  \norm{\mathcal{R}_h(\pi_h^2(\mathbf{w}_\mathcal{N}), \pi_h^3(q_\mathcal{N}))}_{{V^2}' \times {V^3}'}.
\end{equation}
Note that both \acp{nn} are defined in $\mathbb{R}^3$. Besides, $\mathbf{u}_\mathcal{N}$ is not necessarily tangent to $S$. This method is coined with the term trace \acp{feinn} due to its resemblance to TraceFEM~\cite{Olshanskii2017} (as it seeks an approximation of the surface \ac{pde} on a finite-dimensional manifold living in a higher dimensional space). 
However, we stress that, as opposed to TraceFEM, it does not need stabilisation terms in the loss function to restrict the problem and unknowns to the tangent space, as it relies on the interpolation onto surface \ac{fe} spaces defined on $S_h$.
Additionally, although it would be possible to use a single \ac{nn} with $d+1$ output neurons for both the flux and pressure fields, here we use separate \acp{nn} for each field to allow more flexibility in their design, as each solution field may have distinct features. 

Since we are considering a closed surface, we do not need to impose boundary conditions on the \ac{nn} output. For non-empty boundaries, the boundary conditions can be incorporated in the loss function using the interpolant $\pi_{h,0}^2$. 

\subsection{Loss functions} \label{sec:method:loss} 
The \ac{fe} residual $\mathcal{R}_h$ (e.g., the one in~\eqref{eq:maxwell_fe_res}) is isomorphic to the vector $[\mathbf{r}_h(\mathbf{w}_h)]_i = \left< \mathcal{R}_h(\mathbf{w}_h), \varphi^i \right> \doteq \mathcal{R}_h(\mathbf{w}_h)(\varphi^i)$, where $\{\varphi^i\}_{i=1}^N$ are the bases that span ${V_h^i}$, for $i=0,\ldots,3$. Thus, we can use the following loss function:
\begin{equation} \label{eq:maxwell_vec_loss}
    \mathscr{L}(\mathbf{u}_\mathcal{N}) = \norm{\mathbf{r}_{h}(\pi^i_{h,0}(\mathbf{u}_\mathcal{N}))}_\chi,
\end{equation} 
where $\norm{\cdot}_\chi$ is an algebraic norm. The $\ell^2$ norm is often used in the \acp{pinn} literature~\cite{Raissi2019,Kharazmi2021,Berrone2022} and has proven effective in many \ac{feinn} experiments~\cite{Badia2024,Badia2024adaptive}. However, the Euclidean norm of $\mathbf{r}_h$ is a variational crime and this choice is ill-posed in the limit $h \rightarrow 0$;  at the continuous level, the $L^2$-norm of $\mathcal{R}(\mathbf{u})$ is not defined in general.

As proposed in~\cite{Badia2024,Badia2024adaptive}, we can instead use a {\em discrete dual norm} in the loss~\eqref{eq:maxwell_optim} by introducing a discrete Riesz projector $\mathcal{B}_h^{-1}: {V_{h,0}^i}' \to {V_{h,0}^i}$ such that 
\begin{equation*}
    \mathcal{B}_h^{-1}\mathcal{R}_h(\mathbf{w}_h) \in {V_{h,0}^i} \ : \left( \mathcal{B}_h^{-1}\mathcal{R}_h(\mathbf{w}_h), \mathbf{v}_h \right)_{U^i}  = \mathcal{R}_h(\mathbf{w}_h)(\mathbf{v}_h), \quad \forall \mathbf{v}_h \in {V_{h,0}^i},
\end{equation*}
where $\mathcal{B}_h$ is (approximately) the corresponding Gram matrix in ${V_{h,0}^i}$ (the one corresponding to the $H(\bcurl)$ or $H(\bdiv)$ inner product in this work). The $\mathcal{B}_h$-\emph{preconditioned} loss function can be rewritten as
\begin{equation} \label{eq:maxwell_precond_loss}
  \mathscr{L}(\mathbf{u}_\mathcal{N}) = \norm{\mathcal{B}_h^{-1}\mathcal{R}_h (\pi^i_{h,0}(\mathbf{u}_\mathcal{N}))}_\Upsilon,
\end{equation}  
The dual norm of the residual in ${V_{h,0}^i}'$ is attained by using $\Upsilon = U^i$ (i.e., or $H(\bcurl; \Omega)$ or $H(\bdiv; \Omega)$ norm). This is the case that is needed for obtaining optimal error estimates in~\cite{Badia2024adaptive}. However, other choices, e.g., $L^2(\Omega)^d$ are possible.  
In practice, one can consider any spectrally equivalent approximation of $\mathcal{B}_h$, a preconditioner, to reduce computational costs; e.g. the \ac{gmg} preconditioner is effective for the Poisson equation, as shown in~\cite{Badia2024}. Effective multigrid solvers for $H(\bcurl)$ and $H(\bdiv)$ have been designed (see, e.g.,~\cite{Hiptmair2007-lb}), and can be used to reduced the computational cost of the proposed approach.   

After interpolation, $\pi^1_{h,0}(\mathbf{u}_\mathcal{N}) \in U^1_{h,0}$ vanishes on $\Gamma_D$, ensuring that the \ac{feinn} solution $\pi^1_{h,0}(\mathbf{u}_\mathcal{N}) + \bar{\mathbf{u}}_h$ strongly satisfies the Dirichlet boundary conditions. In \ac{fem}, $\bar{\mathbf{u}}_h$ is a standard component and can be computed efficiently. Our method provides a simple and effective way to strongly enforce these conditions, avoiding the need for penalty terms in the loss function~\cite{Raissi2019}, or the use of distance functions~\cite{Sukumar2022, Berrone2022}.
Additionally, Gaussian quadrature rules allow efficient and exact computation of the integrals in~\eqref{eq:maxwell_precond_loss}. Spatial derivatives are now calculated using the polynomial \ac{fe} basis functions, which are more efficient than hybrid automatic differentiation (mixed derivatives with respect to parameters and inputs) widely used in standard \ac{pinn} methods.
Furthermore, the use of a residual dual function (i.e., the inclusion of preconditioners) further ensures well-posedness of the problem and accelerates training convergence.

The formulation can readily be extended to multifield problems. For the Darcy problem above, the loss defined by the residual vector is
\begin{equation} \label{eq:darcy_vec_loss}
    \mathscr{L}(\mathbf{u}_\mathcal{N}, p_\mathcal{N}) = \norm{\mathbf{r}_h(\pi_h^2(\mathbf{u}_\mathcal{N}), \pi_h^3(p_\mathcal{N}))}_\chi.
\end{equation}
This simple yet effective loss function is adopted in the numerical experiments discussed below.
We can also consider the multifield loss function   
\begin{equation} \label{eq:darcy_precond_loss}
    \mathscr{L}(\mathbf{u}_\mathcal{N}, p_\mathcal{N}) = \norm{\mathcal{B}_h^{-1} \mathcal{R}_h (\pi_h^2(\mathbf{u}_\mathcal{N}), \pi_h^3(p_\mathcal{N}))}_{U^2\times U^3},
\end{equation}
where $\mathcal{B}_h^{-1}$ is the Gram matrix in ${V_h^2} \times {V_h^3}$. 

\subsection{Extension to inverse problems} \label{sec:method:inverse}
Given, e.g., partial or noisy observations of the states, inverse problem solvers aim to estimate unknown model parameters and complete the state fields. In~\cite{Badia2024}, the use of \acp{feinn} for solving inverse problems has been extensively discussed. We extend the idea to compatible \acp{feinn} in this section.
We denote the unknown model parameters as $\pmb{\Lambda}$, which may include physical coefficients, forcing terms, etc. We approximate $\pmb{\Lambda}$ with one or several \acp{nn} denoted by $\pmb{\Lambda}_\mathcal{N}$. For exact integration, we introduce the interpolants $\pmb{\pi}_h$ to interpolate $\pmb{\Lambda}_\mathcal{N}$ onto \ac{fe} spaces. 

Let us consider a discrete measurement operator $\mathcal{D}_h: U_h^i \rightarrow \mathbb{R}^{M\times d}$ and the observations $\mathbf{d} \in \mathbb{R}^{M\times d}$, where $M \in \mathbb{N}$ is the number of measurements. 
The loss function for the inverse problem reads:
\begin{equation} \label{eq:inverse_loss}
  \mathscr{L}(\pmb{\Lambda}_\mathcal{N}, \mathbf{u}_\mathcal{N}) \doteq \norm{\mathbf{d} - \mathcal{D}_h (\pi_{h,0}^i(\mathbf{u}_\mathcal{N}) + \bar{\mathbf{u}}_h)}_{\ell^2} + \alpha \norm{\mathcal{R}_h(\pmb{\pi}_h(\pmb{\Lambda}_\mathcal{N}), \pi^i_{h,0}(\mathbf{u}_\mathcal{N}))}_\Upsilon, 
\end{equation}
where $\alpha \in \mathbb{R}^+$ is a penalty coefficient for the \ac{pde} constraint.

\subsection{Numerical analysis} \label{sec:analysis}
The analysis of \acp{feinn} has been carried out in~\cite{Badia2024adaptive} in a general setting. The analysis of compatible \acp{feinn} fits the same framework. We summarise below the main results of this analysis. The well-posedness of the continuous problem relies on the inf-sup condition: there exists a constant $\beta_0 > 0$ such that
\begin{equation}\label{eq:continuous-infsup}
  \inf_{\mathbf{w} \in U} \sup_{\mathbf{v} \in U} \frac{a(\mathbf{w},\mathbf{v})}{\norm{\mathbf{w}}_U \norm{\mathbf{v}}_U} \geq \beta_0,
\end{equation} 
where $U$ is the functional space where the solution of the \ac{pde} is sought (we omit the $i$ superscript in this section). Besides, the bilinear form must be continuous, i.e., there exists a constant $\gamma > 0$ such that
\begin{equation}\label{eq:continuity}
a(\mathbf{u} , \mathbf{v}) \leq \gamma \norm{\mathbf{u}}_U \norm{\mathbf{v}}_U, \quad \forall \mathbf{u},\mathbf{v} \in U.
\end{equation}
Using a Petrov-Galerkin discretisation, the discrete problem is well-posed if a discrete inf-sup condition is satisfied: there exists a constant $\beta > 0$ independent of the mesh size $h$ such that  
\begin{equation}\label{eq:discrete-infsup}
  \inf_{\mathbf{w}_h \in U_h} \sup_{\mathbf{v}_h \in {V_h}} \frac{a(\mathbf{w}_h,\mathbf{v}_h)}{\norm{\mathbf{w}_h}_U \norm{\mathbf{v}_h}_U} \geq \beta, \qquad \beta > 0,
\end{equation} 
where $U_h \subset U$ and ${V_h} \subset U$ are the discrete trial and test spaces, respectively. The discrete inf-sup condition, using $k_U = k_V$, is satisfied for the compatible \ac{fe} spaces being used in this work, which is a direct consequence of the cochain projection between the continuous and discrete de Rham complexes. 
Using other \ac{fe} spaces, such as Lagrange elements, fails the discrete inf-sup condition, making the discrete problem ill-posed. In such cases, regularisation is required to ensure a well-posed \ac{fe} problem. See~\cite{Costabel2002,Badia2012-xf,Badia2009-yj} for more details.
For the Petrov-Galerkin discretisation with linearised test spaces, i.e., $k_U > k_V = 1$, we are not aware of any results in the literature. However, the experimental evaluation of condition numbers for the spaces being considered in this paper indicate that this inf-sup condition is satisfied for all the orders being considered in this work. Under these conditions, the following \emph{nonlinear Cea's lemma} holds, where we assume that the \ac{nn} can emulate the \ac{fe} space for simplicity (see~\cite{Badia2024adaptive} for a more general statement). 
\begin{theorem}\label{eq:feinn-error}
  For any $\epsilon > 0$, the minimum $\mathbf{u}_{\mathcal{N}}$ of~\eqref{eq:maxwell_optim} holds the following a priori error estimate:
  \begin{equation}
\norm{ \mathbf{u} - \pi_h(\mathbf{u}_{\mathcal{N}}) }_U 
 \leq \rho \inf_{w_h \in \mathcal{N}_h} \norm{ \mathbf{u} - \mathbf{w}_h }_U.
  \end{equation}
\end{theorem}
We note that this result is strong because it shows that the minimum of the loss function of the \ac{fe} residual is quasi-optimal on the \ac{nn} manifold, thus exploiting the nonlinear power of this approximation. However, the proof only bounds the error of the \ac{fe} interpolation of the \ac{nn} solution, not the error of the \ac{nn} itself, which is experimentally observed to be much lower in most cases (see~\cite{Badia2024adaptive} for a discussion on this topic).

\section{Numerical experiments} \label{sec:experiments}

In this section, we conduct two comprehensive sets of numerical experiments for forward and inverse problems involving Maxwell's and Darcy equations, the latter being posed on the sphere. In the forward experiments, we compare the performance of compatible \acp{feinn} against \ac{fem}. Our results show that compatible \acp{feinn} achieve higher accuracy than \ac{fem} when the solution is smooth. The use of weak residual discrete dual norms (built out of the inverse of a Gram matrix, or an approximation to it, i.e., a preconditioner) further improves the stability and performance of compatible \acp{feinn}. 
However, the accuracy improvements can hardly overcome the extra cost of the training process compared to \ac{fem}. The use of a fixed mesh for the \ac{nn} interpolation prevents the effective exploitation of \ac{nn} nonlinear approximation properties. For this reason, we consider in a second stage the use of adaptive \acp{feinn} and inverse problems. 

First, we use $h$-adaptive \acp{feinn} to attack problems with localised features, demonstrating that the trained \acp{nn} outperform \ac{fem} solutions for sharp-feature analytic solutions and match \ac{fem} accuracy for singular solutions, all without nested loops.
In the inverse problem experiments, we focus on recovering the unknown physical parameters of Maxwell's equations from partial or noisy observations, using the training strategy proposed in~\cite{Badia2024}. 
We compare the performance of compatible \acp{feinn} with the adjoint \ac{nn} method. With partial observations, the addition of an extra \ac{nn} in the compatible \ac{feinn} method improves the accuracy of both the predicted parameters and recovered states compared to the adjoint \ac{nn} method. When the observations are noisy, compatible \acp{feinn} achieve comparable performance to the adjoint \ac{nn} method {\em without any regularisation terms in the loss function}.

We omit the comparison of the \ac{feinn} method with other \ac{nn}-based approaches, such as \acp{pinn} and \acp{vpinn}. This comparison is already studied in detail in~\cite{Berrone2022,Badia2024}. These studies find that the interpolated \ac{nn} approach performs similar to or better than \acp{pinn} or \acp{vpinn} for a given number of \ac{nn} evaluations. Additionally, the computational cost of interpolated \acp{nn} is lower than that of standard \ac{pinn} methods, as spatial derivatives of \acp{nn} are handled using polynomial interpolation. Furthermore, as demonstrated in~\cite{Badia2024,Badia2024adaptive}, \ac{feinn} convergence can be significantly improved by using preconditioned losses.

\subsection{Implementation} \label{sec:implementation}
The implementation details of \acp{feinn} for forward and inverse problems with weak solutions in $H^1$ are provided in~\cite{Badia2024}, and those for $h$-adaptive \acp{feinn} are included in~\cite{Badia2024adaptive}. In this section, we briefly discuss the additional features required for compatible and trace \acp{feinn}, specifically the interpolation of \acp{nn} onto N{\'e}d{\'e}lec or \ac{rt} \ac{fe} spaces.

We implement the compatible \acp{feinn} method using the Julia programming language. The compatible \ac{fem} part is handled by the \texttt{Gridap.jl}~\cite{Badia2020, Verdugo2022} package, and the \ac{nn} part is implemented using the \texttt{Flux.jl}~\cite{Innes2018} package. Loss function gradient propagation is managed through user-defined rules with the \texttt{ChainRules.jl}~\cite{ChainRules2024} package. For adaptive \acp{feinn}, we rely on the \texttt{GridapP4est.jl}~\cite{GridapP4estGithub} package to handle forest-of-octrees meshes.

The implementation of \ac{nn} interpolations onto N{\'e}d{\'e}lec or \ac{rt} \ac{fe} spaces is as follows. For each moment-based \ac{dof} (e.g., the integral of the normal component of the network over a $S_h$ cell edge for lowest-order \ac{rt} \acp{fe} on the discrete manifold), we evaluate the \ac{nn} at the quadrature points and multiply the results by the corresponding weights to compute the \ac{dof} value. Since, for a fixed mesh, the quadrature points and weights are fixed for each \ac{dof}, they can be extracted and stored once, and reused thereafter at each training iteration. For non-conforming \ac{fe} spaces, we only need to interpolate \acp{nn} onto the master \acp{dof} and use multi-point constraints to compute the slave \acp{dof} values~\cite{Olm2019,Badia2020b}.

\subsection{Forward problems} \label{subsec:forward_exp}
To evaluate the accuracy of the identified solutions $p^{id}$ and $\mathbf{u}^{id}$ for forward problems, we compute their $L^2$ and $H(\bcurl)$ or $H(\bdiv)$ error norms:
\begin{eqnarray*}
    e_{L^2(\Omega)}(p^{id}) = \ltwonorm{p - p^{id}}, \quad
    e_{L^2(\Omega)^d}(\mathbf{u}^{id}) = \ltwonormd{\mathbf{u} - \mathbf{u}^{id}}, \\
    e_{H(\bcurl; \Omega)}(\mathbf{u}^{id}) = \hcurlnorm{\mathbf{u} - \mathbf{u}^{id}}, \quad
    e_{H(\bdiv; \Omega)}(\mathbf{u}^{id}) = \hdivnorm{\mathbf{u} - \mathbf{u}^{id}}
\end{eqnarray*}
where $p: \Omega \rightarrow \mathbb{R}$ and $\mathbf{u}: \Omega \rightarrow \mathbb{R}^d$ are the true states. 
We use sufficient Gauss quadrature points for the integral evaluation to guarantee the accuracy of the computed errors. To simplify the notation, if there is no ambiguity, we omit the domain $\Omega$ and dimension $d$ in the following discussions.

In all forward problem experiments, we use the same \ac{nn} architecture as in~\cite{Badia2024}, i.e., 5 hidden layers with 50 neurons each, and $\texttt{tanh}$ as activation function. Unless otherwise specified, the default training losses are~\eqref{eq:maxwell_vec_loss} and~\eqref{eq:darcy_vec_loss} with the $\chi = \ell^2$ norm. We explicitly state when preconditioned losses~\eqref{eq:maxwell_precond_loss} and~\eqref{eq:darcy_precond_loss} are used, along with the type of preconditioners and dual norms.
Besides, the test \ac{fe} spaces we employ are always linearised, i.e., lowest-order \ac{fe} basis built upon refinement of the trial \ac{fe} mesh if the order of the trial space is not the lowest possible.

In addition to measuring the performance of the interpolated \acp{nn}, we are also interested in how well the \acp{nn} satisfy the Dirichlet boundary condition. We note that Dirichlet boundary condition is implicitly enforced to the \ac{nn} via the \ac{fe} residual minimisation.  
After the training process, in which the interpolated \ac{nn} is used to compute the loss and gradients, we obtain a \ac{nn} solution $\mathbf{u}_\mathcal{N}$. With this solution, we can evaluate the error of the interpolated \ac{nn}, i.e., $\pi_{h,0}(\mathbf{u}_\mathcal{N}) + \bar{\mathbf{u}}_h - \mathbf{u}$, or the \ac{nn} itself, i.e., $\mathbf{u}_{\mathcal{N}} - \mathbf{u}$, where $\mathbf{u}$ is the true solution. In each plot within this section, the label ``\ac{feinn}'' denotes results for the interpolation of the \ac{nn} solutions, while the label ``\ac{nn}'' indicates results for the \acp{nn} themselves (using very accurate quadratures to approximate integrals of \acp{nn}).
For reference, we display the \ac{fem} solution using Petrov-Galerkin discretisation in the plots, labelled as ``\ac{fem}''.

For all the forward and inverse problem experiments, we use the Glorot uniform method~\cite{Glorot2010} to initialise \ac{nn} parameters and the BFGS optimiser in \texttt{Optim.jl}~\cite{Optimjl2018} to train the \acp{nn}. 

\subsubsection{Maxwell's equation with a smooth solution} \label{subsubsec:maxwell_smooth}

For the first set of experiments, we consider the forward Maxwell problem with a smooth solution. We adopt most of the experimental settings from~\cite[Sect. 6.1]{Olm2019}, where the domain is the unit square $\Omega = [0,1]^2$ with a pure Dirichlet boundary ($\Gamma_D = \partial \Omega$). 
We pick $\mathbf{f}$ and $\mathbf{g}$ such that the true state is:
\begin{equation*}
    \mathbf{u}(x,y) = \begin{bmatrix}
        \cos(4.6x)\cos(3.4y) \\
        \sin(3.2x)\sin(4.8y)
    \end{bmatrix}.
\end{equation*}

We first fix the order of the trial basis functions, i.e., $k_U$ is fixed, and refine the mesh to study the convergence of the errors. Considering that the initialisation of \ac{nn} parameters may affect the results, at each mesh resolution, we repeat the experiment 10 times with different \ac{nn} initialisations. 
The results are shown in \fig{fig:fwd_smooth_maxwell_h_refin}. The dashed lines in the plots represent \ac{fem} errors. Different markers distinguish the results of \acp{nn} and their interpolations, and different colours indicate results for different $k_U$. \ac{fem} results exhibit the expected Galerkin's \ac{fem} theoretical error convergence rate 
of $O(h^{k_U})$ measured both in $L^2$ and $H(\bcurl)$ norms.

\begin{figure}[ht]
    \centering
    \begin{subfigure}{0.49\textwidth}
        \includegraphics[width=\textwidth]{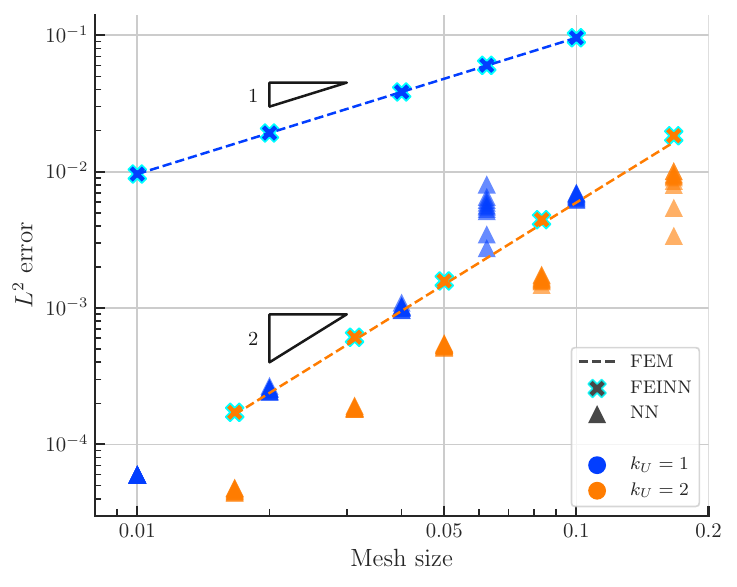}
        \caption{}
        \label{fig:fwd_smooth_maxwell_h_refin_l2_err}
    \end{subfigure}
    \begin{subfigure}{0.49\textwidth}
        \includegraphics[width=\textwidth]{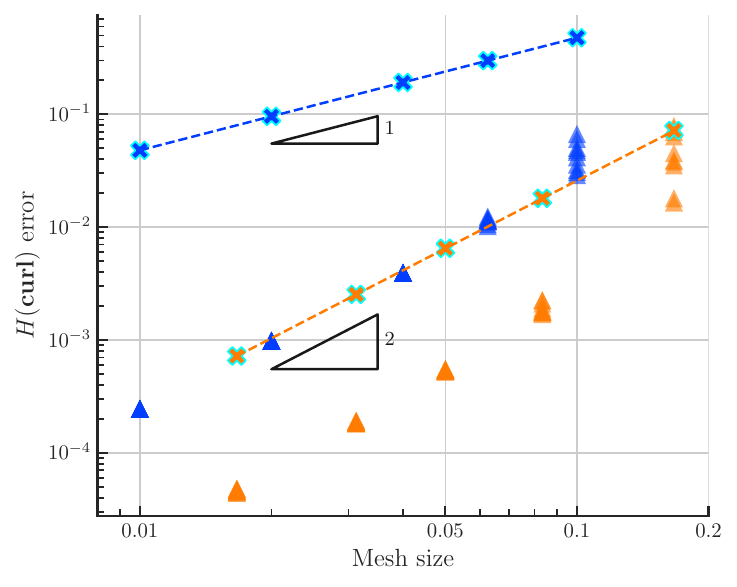}
        \caption{}
        \label{fig:fwd_smooth_maxwell_h_refin_hcurl_err}
    \end{subfigure}
     
    \caption{Error convergence of \ac{feinn} and \ac{nn} solutions with respect to the mesh size of the trial \ac{fe} space for the forward Maxwell problem with a smooth solution. Different colours represent different orders of trial bases.}
    \label{fig:fwd_smooth_maxwell_h_refin}
\end{figure}

Since the variance for \ac{feinn} results with the same mesh size and $k_U$ is negligible, we display the mean values of these 10 runs.
We observe that both the $L^2$ and $H(\bcurl)$ errors of \acp{feinn} consistently fall on the corresponding \ac{fem} lines for both $k_U=1$ and $k_U=2$, indicating successful training and accurate \ac{feinn} solutions. 
For the \acp{nn}, the errors are always below the corresponding \ac{fem} lines, suggesting that the \acp{nn} not only accurately satisfies the Dirichlet boundary condition but also provide more accurate results than the \ac{fem} solutions. 
For example, in \fig{fig:fwd_smooth_maxwell_h_refin_l2_err}, with $h=0.01$ and $k_U=1$, the \ac{nn} is more than 2 orders of magnitude more accurate than the \ac{fem} solution
Similar observations are found for the $H(\bcurl)$ error in \fig{fig:fwd_smooth_maxwell_h_refin_hcurl_err}. 
These findings extend the conclusion in~\cite{Badia2024} that \acp{nn} can outperform \ac{fem} solutions for the Poisson equation with smooth solutions to the Maxwell's equations.
For $k_U=2$, the difference between the \ac{nn} and the \ac{fem} results is less pronounced than for $k_U=1$, but \acp{nn} still beat the \ac{fem} solutions. 
Besides, with a fine enough mesh, and $k_U=1$, it is remarkable to find that the convergence rate of the \acp{nn} matches the $k_U=2$ \ac{fem} rate. This suggests that \acp{nn} can bring superconvergence in certain scenarios.
The improved accuracy of the \acp{nn} compared to the \ac{fem} solutions is likely attributed to the smoothness of the \ac{nn} solutions: the $\mathcal{C}^{\infty}$ NN solution (with the $\texttt{tanh}$ activation function) outperforms the $\mathcal{C}^0$ \ac{fem} solution when approximating a smooth true state.

\subsubsection{The effects of preconditioning on Maxwell's equation} \label{subsubsec:precond}
It is well-known that the condition number of the matrix arising from the \ac{fe} discretisation of Maxwell's problem grows with $k_U^2$ and $h^{-2}$. From preliminary experiments, we observe that \ac{nn} training becomes more difficult as $k_U$ increases, which is somehow expected by the condition number bounds, resulting in a more challenging optimisation problem.
In this experiment, we explore the possibility of using the dual norm of the residual, i.e., the preconditioned loss~\eqref{eq:maxwell_precond_loss}, to reduce the optimisation difficulty and improve training efficiency and accuracy.
We consider the same problem as in \sect{subsubsec:maxwell_smooth}, with $h=2^{-4}$ and $k_U=4$. 
We employ the $\mathbf{B}_{\mathrm{lin}}$ preconditioner, defined as the Gram matrix of the $H(\bcurl)$ inner product on the linearised \ac{fe} space (used as both trial and test spaces). Alternative preconditioners, e.g., the \ac{gmg} preconditioner for Maxwell's equations~\cite{Hiptmair1998}, could also be considered.

In \fig{fig:fwd_precond_maxwell_err_history}, we juxtapose the $H(\bcurl)$ error histories of \acp{nn} and their interpolations during training at the second step, using the dual norm (with preconditioning) and the Euclidean norm (without preconditioning) of the residual vector. To further improve training efficiency, we explore a two-step training strategy, in which we compute the \ac{fe} solution at the interpolation quadrature points as training data, and initialise the \acp{nn} by performing a data fitting task with this data. Next, we train the \acp{nn} with the preconditioned \ac{pde} loss~\eqref{eq:maxwell_precond_loss}.
Note that the first step is computationally much cheaper than the second. An initialisation step with 4,000 iterations is applied to the \acp{nn} before the \ac{pde} training. As a reference, we also display the results without initialisation and preconditioning (labelled as ``0k + N/A'' in the figure). We use the $H(\bcurl)$ (dual) norm and the $L^2$ norm of the Riesz representative of the residual in the preconditioned loss~\eqref{eq:maxwell_precond_loss}. 

The initialisation step significantly accelerates the convergence of the training process, as indicated by the lower initial errors in \fig{fig:fwd_precond_maxwell_err_history}.
Besides, we observe that unpreconditioned \ac{nn} errors spike at the beginning, while preconditioned \acp{nn} maintains a stable error reduction throughout the entire second step. This demonstrates that the well-defined preconditioned loss provides more stable training compared to the residual vector based loss.

More interestingly, as shown in \fig{fig:fwd_precond_maxwell_nn_err_history}, with the $H(\bcurl)$ norm, the \ac{nn} solution starts outperforming the \acp{fem} solution after around 1,500 iterations and achieves nearly an order of magnitude lower $H(\bcurl)$ error by 6,000 iterations.
The preconditioned \ac{pde} training serves as a post-processing step to improve \ac{fe} solution accuracy (more than one order of magnitude in the experiments) without increasing the mesh resolution or trial order.
In summary, the initialisation step, while not required, can lower computational cost and provide faster error reduction. Moreover, the $\mathbf{B}_{\mathrm{lin}}$ preconditioner, combined with $L^2$ or $H(\bcurl)$ norms in the loss~\eqref{eq:maxwell_precond_loss}, greatly accelerates \ac{nn} convergence and improves the accuracy of the trained \acp{nn}. These results suggest that these schemes can be effective for transient problems, where the \ac{nn} will be properly initialised with the solution at the previous time step.

\begin{figure}[ht]
    \centering
    \begin{subfigure}{0.49\textwidth}
        \includegraphics[width=\textwidth]{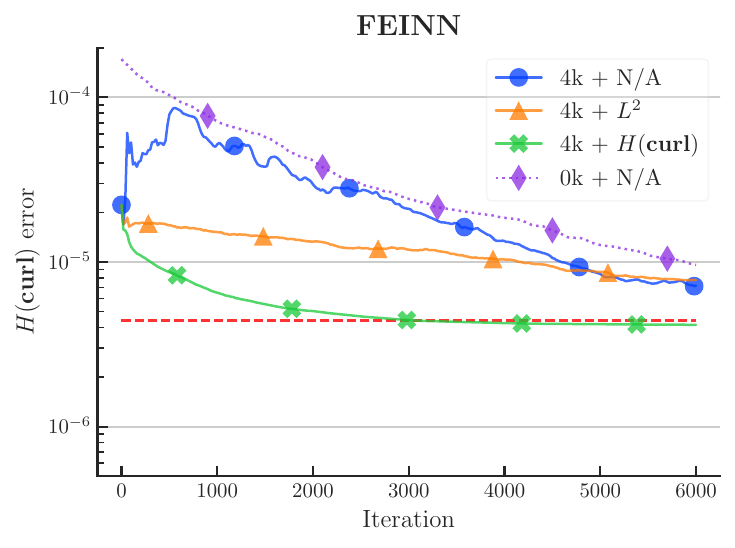}
        \caption{}
        \label{fig:fwd_precond_maxwell_feinn_err_history}
    \end{subfigure}
    \begin{subfigure}{0.49\textwidth}
        \includegraphics[width=\textwidth]{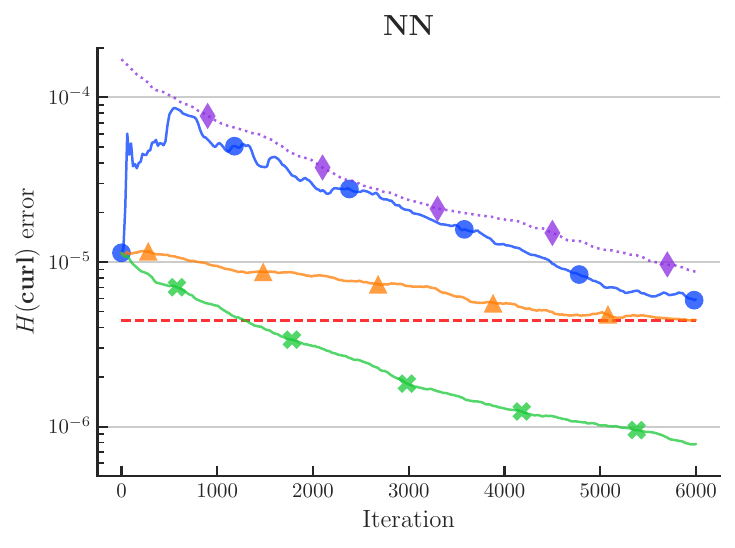}
        \caption{}
        \label{fig:fwd_precond_maxwell_nn_err_history}
    \end{subfigure}
     
    \caption{$H(\bcurl)$ error history of and \acp{feinn} and \acp{nn} during training using different dual norms in the preconditioned loss for the forward Maxwell problem with a smooth solution; ``N/A'' indicates no preconditioning. The number in the legends indicates iterations in the initialisation step. Mesh size $h=2^{-4}$ and order $k_U=4$ were used.}
    \label{fig:fwd_precond_maxwell_err_history}     
\end{figure}

With $\mathbf{u} \in \mathcal{C}^\infty(\bar{\Omega})$ and an effective preconditioner, we can now investigate how the error behaves with increasing $k_U$. 
We revisit the same problem as in \sect{subsubsec:maxwell_smooth}, keeping the mesh size $h$ fixed and increasing $k_U$ from 1 to 4. The errors versus $k_U$ are presented in \fig{fig:fwd_smooth_maxwell_p_refin}. For each order, we also run 10 experiments, and only the mean values of the \ac{feinn} errors are displayed. 
We employ the preconditioner $\mathbf{B}_{\mathrm{lin}}$ equipped with $H(\bcurl)$ dual norm during training.
We investigate two mesh sizes, $h=1/15$ and $h=1/30$, distinguished by different colours in the figure. 
We observe that, similar to the findings reported in~\cite{Badia2024}, the \ac{feinn} errors can still recover the corresponding \ac{fem} errors as $k_U$ increases, and the \ac{nn} errors are always below the \ac{fem} errors.
For instance, when $k_U=3$ and $h=1/30$, the \ac{nn} solutions outperform the \ac{fem} solution by more than one order of magnitude in both $L^2$ and $H(\bcurl)$ errors. 
Besides, preconditioners also enhance the accuracy of \ac{nn} solutions. For $k_U=2$ and $H(\bcurl)$ errors, the preconditioned \ac{nn} solutions, as shown in \fig{fig:fwd_smooth_maxwell_p_refin_hcurl_err}, surpass the \ac{fem} solution by two orders of magnitude. In contrast, without preconditioning, depicted in \fig{fig:fwd_smooth_maxwell_h_refin_hcurl_err}, the \ac{nn} solutions exceed the \ac{fem} solutions by only one order of magnitude.
It is noteworthy to mention that, as $k_U$ increases, the \ac{nn} errors approach the \ac{fem} errors, suggesting a potential need for enrichments in the \ac{nn} architecture. However, studying the effects of different \ac{nn} architectures is beyond the scope of this paper.

\begin{figure}[ht]
    \centering
    \begin{subfigure}{0.49\textwidth}
        \includegraphics[width=\textwidth]{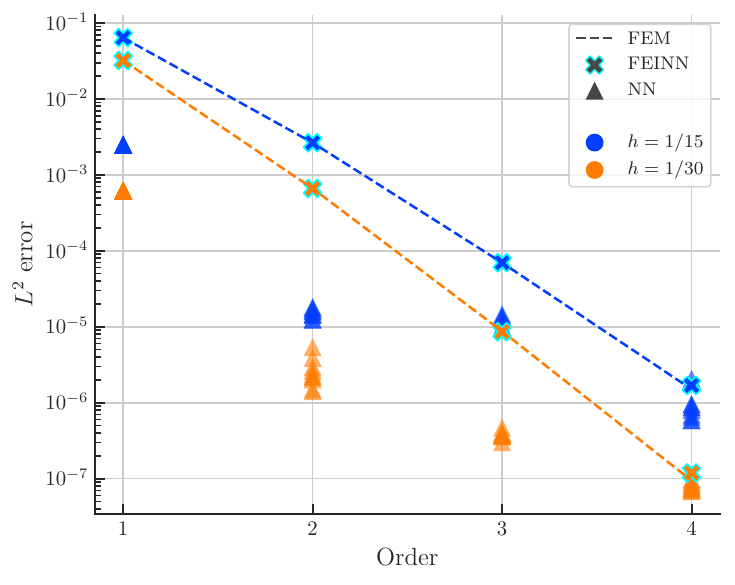}
        \caption{}
        \label{fig:fwd_smooth_maxwell_p_refin_l2_err}
    \end{subfigure}
    \begin{subfigure}{0.49\textwidth}
        \includegraphics[width=\textwidth]{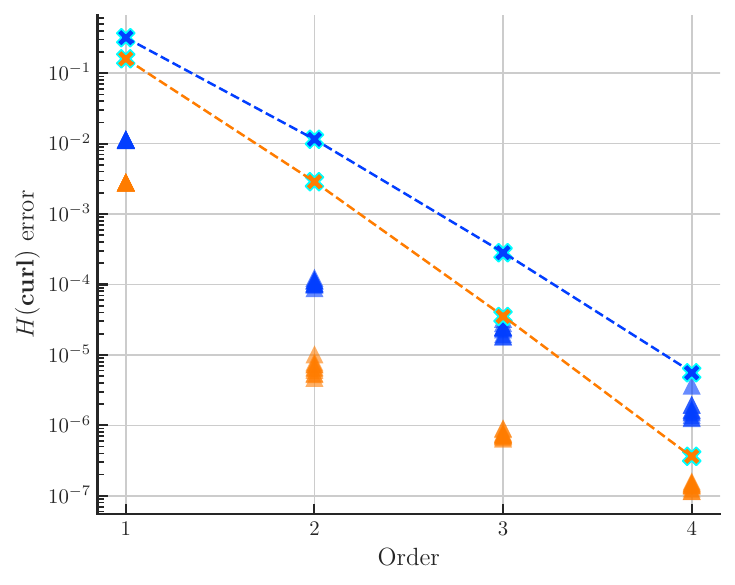}
        \caption{}
        \label{fig:fwd_smooth_maxwell_p_refin_hcurl_err}
    \end{subfigure}
     
    \caption{Error convergence of \ac{feinn} and \ac{nn} solutions with respect to the order of trial bases for the forward Maxwell problem with a smooth solution. Loss function subject to minimisation is equipped with preconditioner $\mathbf{B}_{\mathrm{lin}}$ and the $H(\bcurl)$ norm. Different colours represent different mesh sizes.}
    \label{fig:fwd_smooth_maxwell_p_refin}
\end{figure}

\subsubsection{Darcy equation on a sphere} \label{subsubsec:darcy}
We now attack a Darcy problem defined on the unit sphere $S=\{(x,y,z): x^2+y^2+z^2=1\}$. We refer to \sect{sec:method:tracefeinn} for the trace \ac{feinn} approach to solve this problem.
We adopt the following analytical solutions from~\cite[Sect. 5.1.1]{Rognes2013}:
\begin{equation*}
    \mathbf{u}(x, y, z) = \frac{1}{x^2+y^2+z^2} \begin{bmatrix} y^3z - 2x^2yz + yz^3 \\ x^3z - 2xy^2z + xz^3 \\ x^3y - 2xyz^2 + xy^3 \end{bmatrix}, \quad p(x, y, z) = -xyz.
\end{equation*}
The source term is $f(x, y, z) = -12xyz$. Note that the analytical $p$ satisfies the zero mean property on $S$, i.e., $\int_S p = 0$.

Since the problem has no boundary conditions, given a solution $(\mathbf{u},p)$, another valid solution can be obtained by adding an arbitrary constant to $p$. 
To ensure solution uniqueness, we use a (scalar) Lagrange multiplier to impose the global zero mean constraint.
This approach is widely used in the \ac{fem} community and it ensures that the discretisation matrix is invertible, making the \ac{fe} problem well-posed. 
We then revise the weak formulation in~\eqref{eq:darcy_weak} as follows: find $(\mathbf{u}, p, \lambda) \in U^2 \times \tilde{U}^3 \times \mathbb{R}$ such that
\begin{equation*}
    (\mathbf{u}, \mathbf{v}) - (p, \pmb{\nabla}_S \cdot \mathbf{v}) = 0 \quad \forall \mathbf{v} \in U^2, \quad
    (\pmb{\nabla}_S \cdot \mathbf{u}, q) + (\lambda, q) = (f, q) \quad \forall q \in \tilde{U}^3, \quad
    (p, \sigma) = 0 \quad \forall \sigma \in \mathbb{R}.
\end{equation*}
The corresponding \ac{fe} formulation is similar to~\eqref{eq:darcy_fe}.

To discretise the domain, we employ the so-called cubed sphere~\cite{Ronchi1996}. This mesh is available in the \texttt{GridapGeosciences.jl}~\cite{GridapGeosciencesGithub} package, a \texttt{Gridap.jl}'s extension for the numerical approximation of geophysical flows~\cite{Lee2024}.
In the experiments, each of the six panels of the cubed sphere is discretised with $n_e \times n_e$ quadrilateral elements, where $n_e$ is the number of elements on each side of the panel. Thus, the total number of elements is $6\times n_e\times n_e$. 

As mentioned in \sect{sec:method:tracefeinn}, we use two \acp{nn} $\mathbf{u}_\mathcal{N}$ and $p_\mathcal{N}$ to represent the flux and the pressure, respectively. We use first-order \ac{rt} spaces to interpolate $\mathbf{u}_\mathcal{N}$ and piecewise-constant \ac{dg} spaces to interpolate $p_\mathcal{N}$. 
The \ac{nn} interpolation onto \ac{dg} spaces is the same as that onto \ac{cg} spaces covered in~\cite{Badia2024,Badia2024adaptive}, which is just defined at the evaluation of $p_\mathcal{N}$ at the nodes of the \ac{fe} space. 

\fig{fig:fwd_darcy_results} shows the interpolated \ac{nn} solutions for the flux, its divergence, and the pressure on a cubed sphere mesh consisting of $6\times50\times50$ quadrilateral elements.
Note that the flux solution in \fig{fig:fwd_darcy_feinn_flux} is tangential to the sphere by construction of the \ac{rt} space.  In \fig{fig:fwd_darcy_feinn_div_flux}, the flux divergence displays discontinuities at element boundaries because \ac{rt} bases have $\mathcal{C}^0$ continuous normal components.
\fig{fig:fwd_darcy_feinn_pressure} demonstrates that the pressure solution exhibits a form of symmetry, likely leading to a zero mean on the sphere. 

\begin{figure}[ht]
    \centering
    \begin{subfigure}[t]{0.32\textwidth}
        \includegraphics[width=\textwidth]{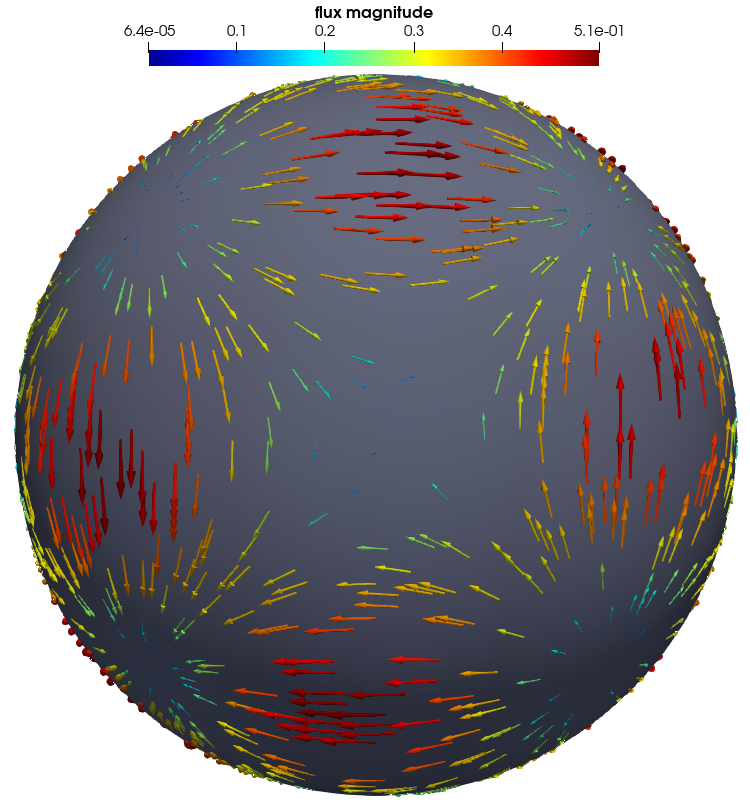}
        \caption{$\mathbf{u}^{id}$}
        \label{fig:fwd_darcy_feinn_flux}
    \end{subfigure}
    \begin{subfigure}[t]{0.32\textwidth}
        \includegraphics[width=\textwidth]{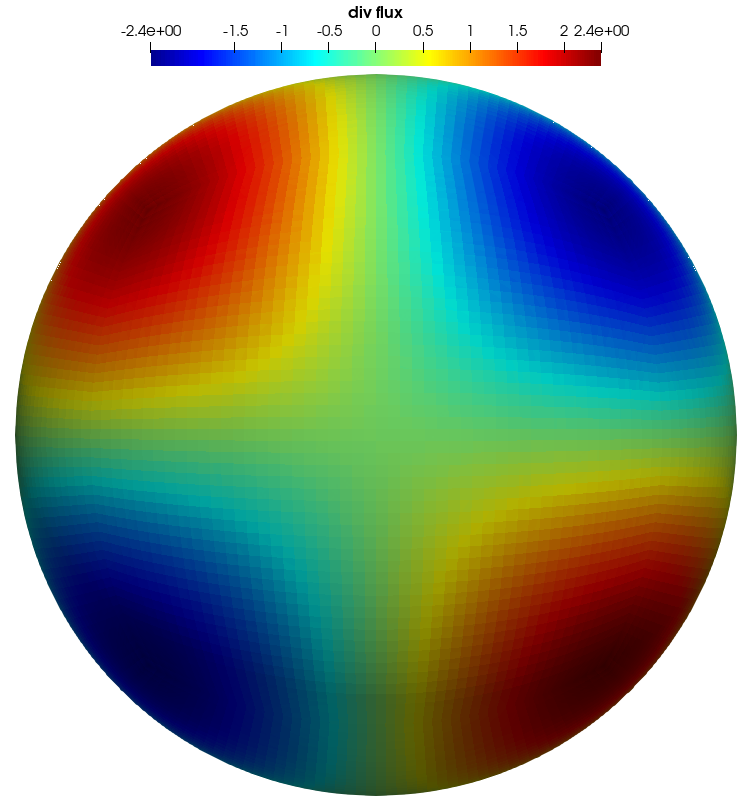}
        \caption{$\pmb{\nabla} \cdot \mathbf{u}^{id}$}
        \label{fig:fwd_darcy_feinn_div_flux}
    \end{subfigure}
    \begin{subfigure}[t]{0.32\textwidth}
        \includegraphics[width=\textwidth]{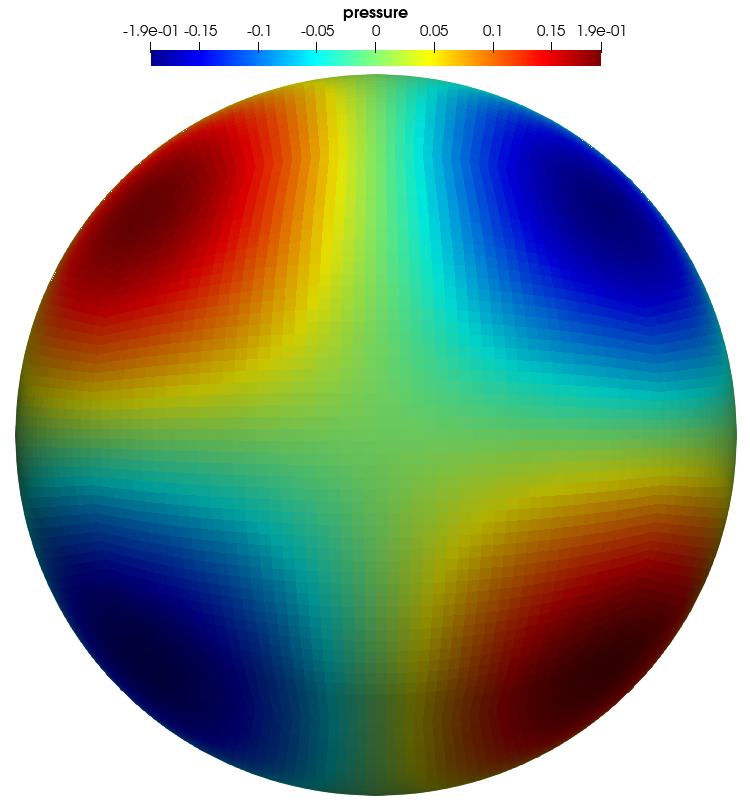}
        \caption{$p^{id}$}
        \label{fig:fwd_darcy_feinn_pressure}
    \end{subfigure}
     
    \caption{Interpolated \ac{nn} solutions on a $6\times50\times50$ grid for the forward Darcy problem on a sphere.}
    \label{fig:fwd_darcy_results}
\end{figure}

Next, we analyse the convergence of the trace \ac{feinn} method with respect to mesh size ($h\approx2/n_e$). 
The errors of the interpolated \acp{nn} are plotted in \fig{fig:fwd_darcy_h_refin_errors}. Again, each marker for \acp{feinn} represents the mean value from 10 independent experiments.
It is noteworthy to mention that, the flux \ac{nn} solution is not constrained to be tangential to the sphere. 
However, its interpolation onto the surface \ac{fe} space (using \ac{rt} elements and Piola maps) belongs to the tangent bundle.
By interpolating the trained \acp{nn} onto a \ac{fe} space with higher-order bases and a finer mesh, we observed that one can achieve more accurate solutions, {\em even if a coarser mesh and order were used for training}. A similar phenomenon occurs in data science, such as image super-resolution~\cite{Dong2014}, where \acp{nn} trained on lower-resolution images can generate higher-resolution outputs. Specifically, we interpolate the trained flux \acp{nn} onto a \ac{fe} space of order 4 and a mesh of $6\times200\times200$ quadrilateral elements.
For convenience, in \fig{fig:fwd_darcy_h_refin_errors} and the rest of the section, we refer to both the trained pressure \acp{nn} and the interpolations of the trained flux \acp{nn} onto a finer mesh as \ac{nn} solutions, labelled as ``\ac{nn}*''.

For the interpolated \acp{nn}, \fig{fig:fwd_darcy_h_refin_errors} shows that they successfully match the \ac{fem} solutions for both flux and pressure across various mesh resolutions, with all the \ac{feinn} markers aligning on their corresponding \ac{fem} lines.
Besides, the \ac{nn} solutions are more accurate than the \ac{fem} solutions, especially for the flux: both $L^2$ and $H(\bdiv)$ errors are more than two orders of magnitude lower than corresponding \ac{fem} errors. Besides, a visual comparison of the slopes in \fig{fig:fwd_darcy_refin_l2_err} and \fig{fig:fwd_darcy_refin_hdiv_err} indicates that \ac{nn} solutions exhibit a higher convergence rate than \ac{fem} solutions.
These findings highlight the potential of \acp{nn} in solving surface \acp{pde} more efficiently by training on a coarse mesh and then interpolating onto a fine mesh for the final solution.

\begin{figure}[ht]
    \centering
    \begin{subfigure}{0.49\textwidth}
        \includegraphics[width=\textwidth]{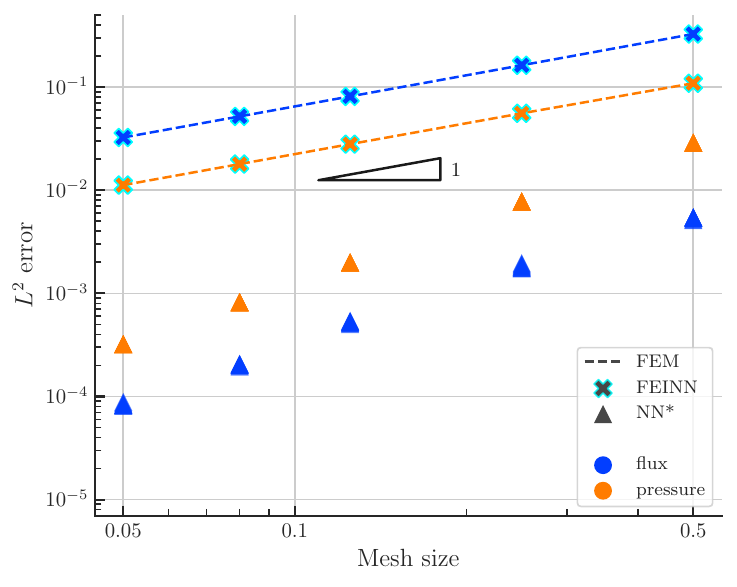}
        \caption{}
        \label{fig:fwd_darcy_refin_l2_err}
    \end{subfigure}
    \begin{subfigure}{0.49\textwidth}
        \includegraphics[width=\textwidth]{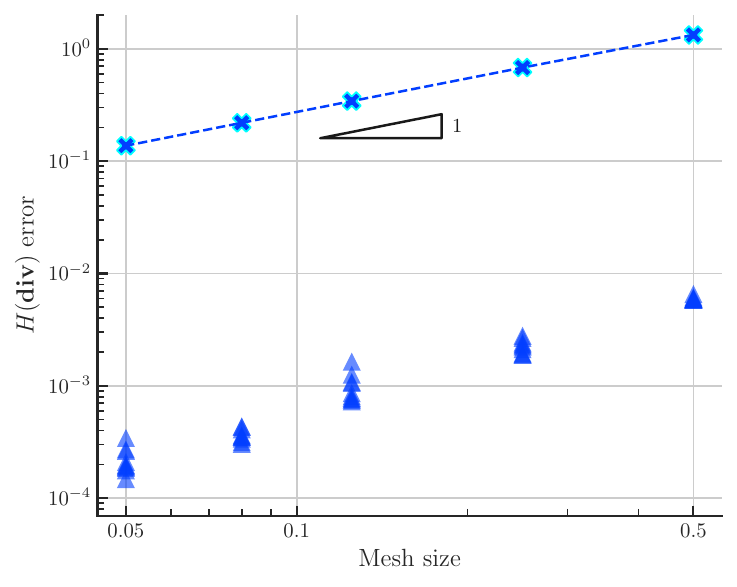}
        \caption{}
        \label{fig:fwd_darcy_refin_hdiv_err}
    \end{subfigure}
     
    \caption{Error convergence of \ac{feinn} and \ac{nn} solutions with respect to the mesh size of the trial \ac{fe} space for the forward Darcy problem on the unit sphere. Colours distinguish between flux and pressure.} 
    \label{fig:fwd_darcy_h_refin_errors}
\end{figure}

\subsubsection{Adaptive training for Maxwell's equation with localised features} \label{subsubsec:adaptive_feinn}
We conclude the forward experiments section by employing the $h$-adaptive \ac{feinn} method proposed in~\cite{Badia2024adaptive} to tackle forward Maxwell problems with localised features. We consider both sharp features and singularities in the solution. We first focus on the following analytic solution defined in the unit square domain $\Omega=[0,1]^2$ with a pure Dirichlet boundary:
\begin{equation*}
    \mathbf{u}(x,y) = \begin{bmatrix}
        \arctan(50(\sqrt{(x + 0.05)^2 + (y + 0.05)^2} - 1.2)) \\
        1.5\exp(-500((x - 0.5)^2 + (y - 0.5)^2))
    \end{bmatrix}.
\end{equation*}
Note that the x-component of $\mathbf{u}$ has a circular wave front centred at $(-0.05, -0.05)$ with a radius of 1.2, while the y-component features a sharp peak at $(0.5, 0.5)$. These sharp features are typically benchmarks for evaluating the performance of numerical methods equipped with adaptive mesh refinement.

We adopt the train, estimate, mark, and refine strategy proposed in~\cite{Badia2024adaptive} to adaptively refine the mesh and train \acp{feinn}.
Preconditioning is applied in the loss function to improve the training stability and to speed up the convergence. 
We use the $L^2$ norm in the preconditioned loss function~\eqref{eq:maxwell_precond_loss} and, similar to~\cite{Badia2024adaptive} and the previous experiment, we utilise the preconditioner $\mathbf{B}_{\mathrm{lin}}$.
We use different error indicators to guide mesh adaptation. Namely, the $L^2$ error among the interpolated \ac{nn} and the analytical solution, the $L^2$ norm of the difference between the \ac{nn} and its \ac{fe} interpolation, and the $L^2$ norm of the strong \acp{pde} residual evaluated at the \ac{nn}. 
These are labelled as ``real'', ``integration'', and ``network'', respectively, in the figures below.
We leverage forest-of-octrees meshes for adaptivity, which are typically non-conforming. To maintain curl-conforming property in the \ac{fe} space, we incorporate additional linear multi-point constraints for slave \acp{dof}. For a more detailed discussion on $h$-adaptive \ac{feinn}, we refer to~\cite{Badia2024adaptive}. 

In the experiments, we fix $k_U=2$, with the initial trial \ac{fe} space consisting of $16\times16$ quadrilateral elements. We refine the mesh up to 7 times, each with a 10\% refinement ratio (i.e., we refine 10\% of the cells with the largest error indicator), and run 20 independent experiments initialised with different \ac{nn} parameters for each error indicator.
The training iterations for each step are $[100, 100, 200, 200, 300, 300, 400, 400]$, totalling 2000 iterations.

In \fig{fig:fwd_adaptive_maxwell_indicator_comparison}, we present the convergence of the \ac{nn} errors with respect to the refinement step. 
\fig{fig:fwd_maxwell_adapt_refin_l2_err} shows errors measured in the $L^2$ norm, while \fig{fig:fwd_maxwell_adapt_refin_hcurl_err} in the $H(\bcurl)$ norm. 
The solid line represents the median of the errors, while the band represents the range from the minimum to the 90th percentile across 20 independent runs. 
The red dashed line illustrates the error of the $h$-adaptive \ac{fem} solution, using the real \ac{fem} $H(\bcurl)$ error as the refinement indicator.

\begin{figure}[ht]
    \centering
    \begin{subfigure}{0.49\textwidth}
        \includegraphics[width=\textwidth]{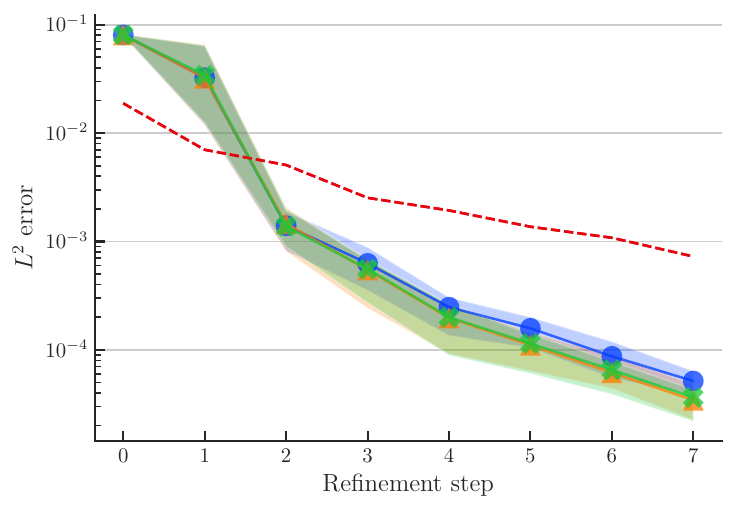}
        \caption{}
        \label{fig:fwd_maxwell_adapt_refin_l2_err}
    \end{subfigure}
    \begin{subfigure}{0.49\textwidth}
        \includegraphics[width=\textwidth]{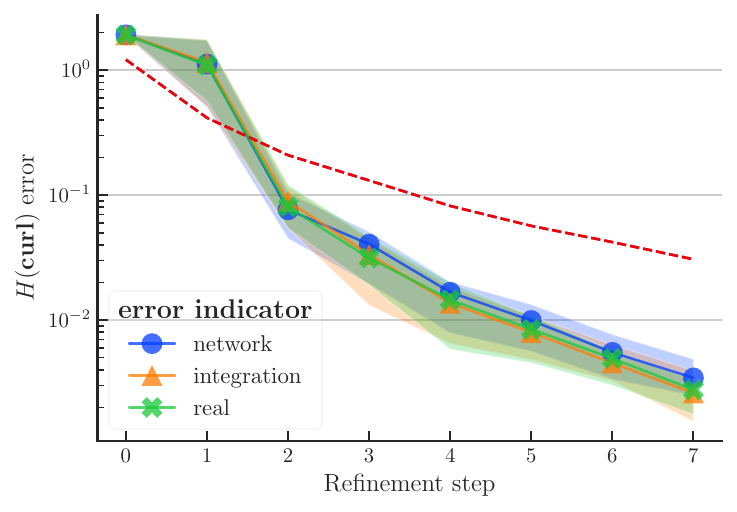}
        \caption{}
        \label{fig:fwd_maxwell_adapt_refin_hcurl_err}
    \end{subfigure}
     
    \caption{Error convergence of \ac{nn} solutions with respect to refinement steps for the forward Maxwell problem with less smooth solution, using different error indicators. The solid line denotes the median, and the band represents the range from the minimum to the 90th percentile across 20 independent runs. The red dashed line illustrates the $h$-adaptive \ac{fem} solution error using the real \ac{fem} $H(\bcurl)$ error as the refinement indicator.} 
    \label{fig:fwd_adaptive_maxwell_indicator_comparison}
\end{figure}

Although the trained \acp{nn} start with higher errors, they begin to outperform the $h$-adaptive \ac{fem} solution after the first two refinement steps. As shown in both \fig{fig:fwd_maxwell_adapt_refin_l2_err} and \fig{fig:fwd_maxwell_adapt_refin_hcurl_err}, the \ac{nn} errors consistently fall below the \ac{fem} error lines after the second refinement.
This observation agrees with the findings in~\cite{Badia2024adaptive}, where \acp{nn} can beat $h$-adaptive \ac{fem} for smooth analytic solutions. 
However, in~\cite{Badia2024adaptive}, \ac{nn} solutions are noted to be superior mainly in $H^1$-error for Poisson equations, but our Maxwell experiments show lower errors in both $L^2$ and $H(\bcurl)$ norms.
Furthermore, the network and integration error indicators effectively guides mesh refinements, as reflected by the closely aligned convergence curves across different indicators. A similar finding was also reported in~\cite{Badia2024adaptive} for the Poisson problems.

Next, we increase the difficulty of the problem by introducing singular feature in the solution. The analytic solution is now defined in the domain $\Omega=[0,1]^2$ as follows:
\begin{equation*}
    \mathbf{u}(x, y) = \pmb{\nabla}(r^{\frac{2}{3}}\sin(\frac{2\theta}{3})), \quad r = \sqrt{x^2 + y^2}, \quad \theta = \arctan(\frac{y}{x}).
\end{equation*}
This solution exhibits a singularity at the origin and serves as a popular benchmark for testing adaptive methods (see, e.g.,~\cite{Olm2019}). Most of the settings remain the same as the previous experiment, except for the following adjustments: the initial mesh comprises $20\times20$ quadrilateral elements, the $H(\bcurl)$ norm is used in the preconditioned loss function, and the training iterations are in total 5,000 with $[500, 750, 1000, 1250, 1500]$ iterations for each step. 

\fig{fig:fwd_adaptive_maxwell_singular_indicator_comparison} displays the convergence of the \ac{nn} errors across refinement step from 20 independent runs. The \ac{nn} errors are computed by interpolating the trained \acp{nn} onto a second-order \ac{fe} space on a finer uniform mesh. Despite the low regularity, the \ac{nn} solutions consistently match the adaptive \ac{fem} solutions at each refinement step for both $L^2$ and $H(\bcurl)$ errors. The figure also highlights that the proposed network and integration error indicators effectively guide mesh refinements for this singular problem, achieving  performance comparable to the real error that is barely available in practice.
Overall, \fig{fig:fwd_adaptive_maxwell_indicator_comparison} and \fig{fig:fwd_adaptive_maxwell_singular_indicator_comparison} illustrate the effectiveness of the $h$-adaptive \ac{feinn} method in solving $H(\bcurl)$ problems with sharp features and singularities. 

\begin{figure}[ht]
    \centering
    \begin{subfigure}{0.49\textwidth}
        \includegraphics[width=\textwidth]{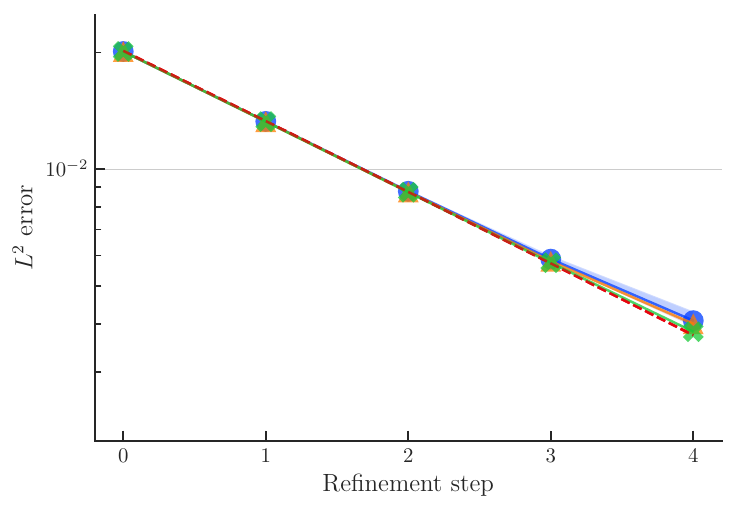}
        \caption{}
        \label{fig:fwd_maxwell_singular_adapt_refin_l2_err}
    \end{subfigure}
    \begin{subfigure}{0.49\textwidth}
        \includegraphics[width=\textwidth]{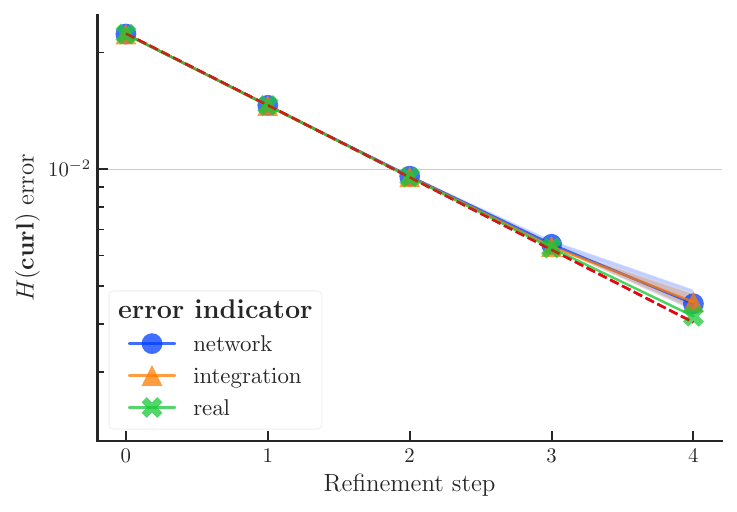}
        \caption{}
        \label{fig:fwd_maxwell_singular_adapt_refin_hcurl_err}
    \end{subfigure}
     
    \caption{Error convergence of \ac{nn} solutions with respect to refinement steps for the forward Maxwell problem with singularity, using different error indicators. See the caption of \fig{fig:fwd_adaptive_maxwell_indicator_comparison} for details about the information displayed in this figure.}
    \label{fig:fwd_adaptive_maxwell_singular_indicator_comparison}
\end{figure}

\subsection{Inverse problems} \label{subsec:inverse_exp}
In this section, we present the results of the inverse Maxwell problems, focusing on situations with partial or noisy observations of the state $\mathbf{u}$ and a fully-unknown coefficient $\kappa$. We compare the performance of compatible \acp{feinn} with adjoint \acp{nn} for both types of observations. 
The adjoint \ac{nn} method was first proposed in~\cite{Berg2021}, and then further explored in~\cite{Mitusch2021}. 
This method approximates the unknown coefficient with a \ac{nn} and uses the adjoint method to compute the gradient of the data misfit with respect to the \ac{nn} parameters.
While previous works~\cite{Berg2021,Mitusch2021,Badia2024} mainly focus on inverse Poisson problems, in this work we extend the comparison to inverse Maxwell's problems.

For the compatible \ac{feinn} method, we employ two \acp{nn}, $\mathbf{u}_\mathcal{N}$ and $\kappa_\mathcal{N}$, to approximate the state and coefficient, respectively. 
The adjoint \ac{nn} method only requires a single \ac{nn}, $\kappa_\mathcal{N}$, to represent the unknown coefficient. To ensure a fair comparison, we use the same \ac{nn} architecture for $\kappa_\mathcal{N}$ in both methods.
In the plots, the label ``Adjoint\ac{nn}'' denotes results from the adjoint \ac{nn} method; the label ``\ac{feinn}'' indicates interpolated \ac{nn} results from the compatible \ac{feinn} method. The label tag ``\ac{nn} only'' for \acp{feinn} represents results associated with \acp{nn} themselves. 
Consistent with~\cite{Badia2024}, we use the $\texttt{softplus}$ activation function, and apply a rectification function $r(x)=|x|+0.01$ as the activation for the output layer of $\kappa_\mathcal{N}$ to ensure its positivity.

We adopt the three-step training strategy proposed in~\cite{Badia2024} for solving inverse problems with \acp{feinn}. 
First, we initialise $\mathbf{u}_\mathcal{N}$ by training it with the observations. Then, we fix $\mathbf{u}_\mathcal{N}$ and train $\kappa_\mathcal{N}$ using only the \ac{pde} loss. Finally, we fine-tune both $\mathbf{u}_\mathcal{N}$ and $\kappa_\mathcal{N}$ with the full loss~\eqref{eq:inverse_loss}.
The final step contains multiple substeps with increasing penalty coefficients for the \ac{pde} term in the loss. We adhere to the convention used in~\cite{Badia2024} to denote the training iterations and penalty coefficients. 
For example, $[100, 50, 3\times400]$ iterations means that the first step involves 100 iterations, the second step has 50 iterations, and the third step comprises three substeps, each with 400 iterations. Similarly, $\alpha=[0.1,0.2,0.3]$ denotes that the penalty coefficients for those substeps are 0.1, 0.2, and 0.3, respectively.

We use the following three relative errors to evaluate the accuracy of the identified state $\mathbf{u}^{id}$ and coefficient $\kappa^{id}$:
\begin{equation*}
    \varepsilon_{L^2(\Omega)^d}(\mathbf{u}^{id}) = \frac{\ltwonormd{\mathbf{u}^{id} - \mathbf{u}}}{\ltwonormd{\mathbf{u}}}, \  
    \varepsilon_{H(\bcurl; \Omega)}(\mathbf{u}^{id}) = \frac{\hcurlnorm{\mathbf{u}^{id} - \mathbf{u}}}{\hcurlnorm{\mathbf{u}}}, \ 
    \varepsilon_{L^2(\Omega)}(\kappa^{id}) = \frac{\ltwonorm{\kappa^{id} - \kappa}}{\ltwonorm{\kappa}},
\end{equation*}
where $\mathbf{u}: \Omega \rightarrow \mathbb{R}^d$ and $\kappa: \Omega \rightarrow \mathbb{R}^+$ are the true coefficient and state, respectively.

\subsubsection{Partial observations} \label{subsubsec:inv_partial_obs}
In this section, we tackle an inverse Maxwell problem with partial observations of the state. The problem is defined on the unit square $\Omega=[0,1]^2$, with pure Dirichlet boundary conditions. We assume that only the $x$-component of the magnetic field $\mathbf{u}$ is observed, with a total of $70^2$ observations distributed uniformly within the region $[0.005,0.995]^2$. 
This region selection avoids observations on the Dirichlet boundary, as the state is fully known there.
In addition, one-component observations of a vector-valued field are common in practice~\cite{Oberai2003}.
Note that with too few observations, the problem becomes highly ill-posed, leading to identified solutions that diverge significantly from the true ones for both methods. 
The analytical state $\mathbf{u}$ (\fig{fig:partial_obs_true_state}) and coefficient $\kappa$ (\fig{fig:partial_obs_true_coeff}) are defined as follows:
\begin{equation*}
    \mathbf{u}(x, y) = \begin{bmatrix} \cos(\pi x)\cos(\pi y) \\ \sin(\pi x)\sin(\pi y) \end{bmatrix}, \quad \kappa(x, y) = 1 + 5\mathrm{e}^{-5((2x - 1)^2 + (y - 0.5)^2)}.
\end{equation*}

\begin{figure}[ht]
    \centering
    \begin{subfigure}[t]{0.32\textwidth}
        \includegraphics[width=\textwidth]{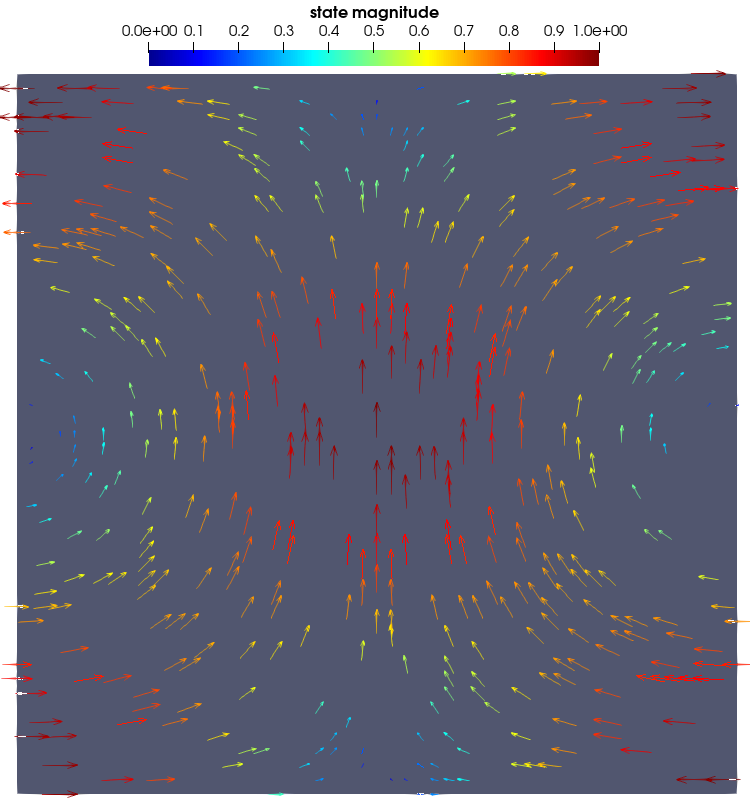}
        \caption{$\mathbf{u}$}
        \label{fig:partial_obs_true_state}
    \end{subfigure}
    \begin{subfigure}[t]{0.32\textwidth}
        \includegraphics[width=\textwidth]{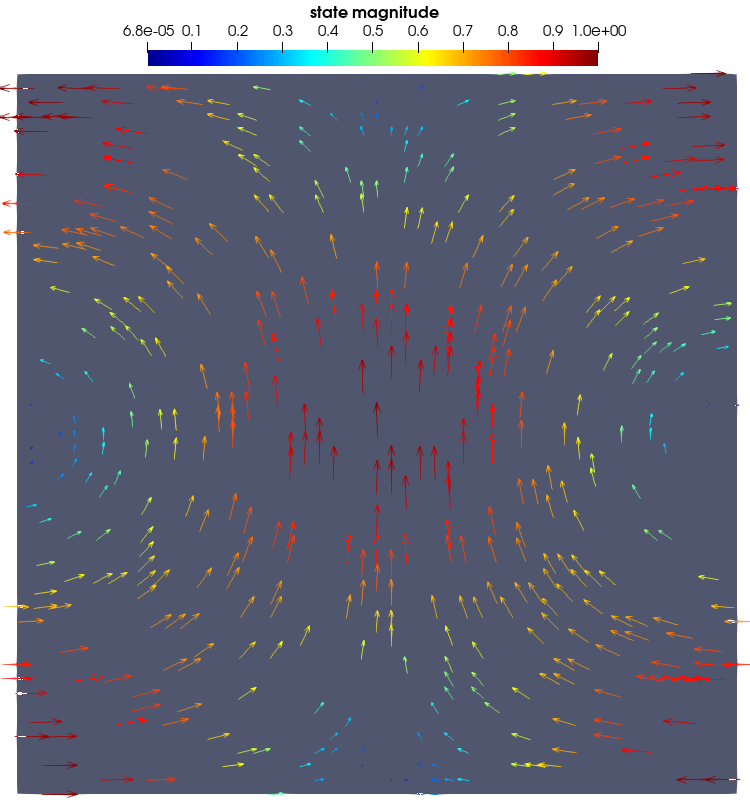}
        \caption{$\mathbf{u}^{id}$}
        \label{fig:partial_obs_feinn_state}
    \end{subfigure}
    \begin{subfigure}[t]{0.32\textwidth}
        \includegraphics[width=\textwidth]{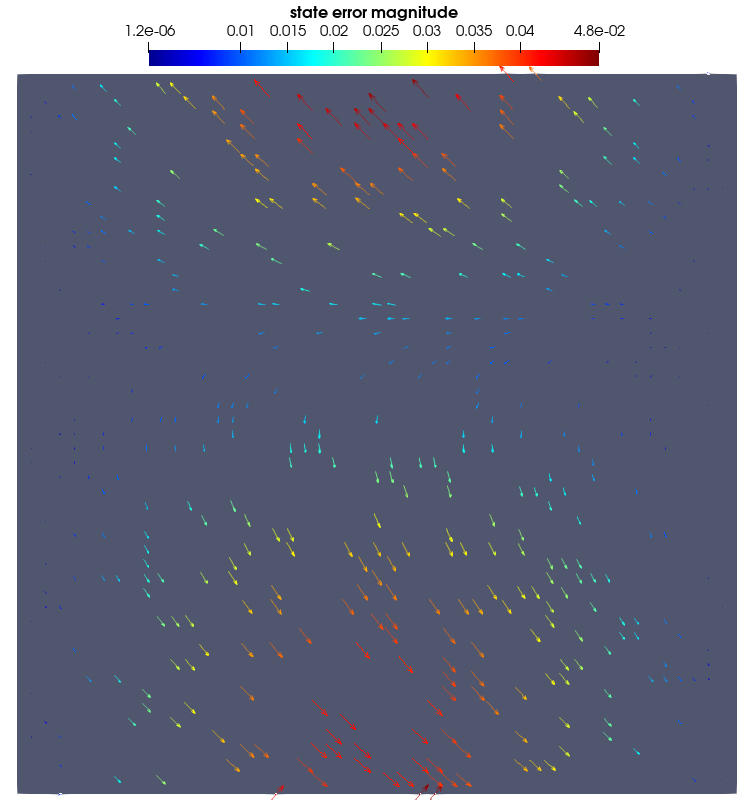}
        \caption{$\mathbf{u}^{id} - \mathbf{u}$}
        \label{fig:partial_obs_state_error}
    \end{subfigure}

    \centering
    \begin{subfigure}[t]{0.32\textwidth}
        \includegraphics[width=\textwidth]{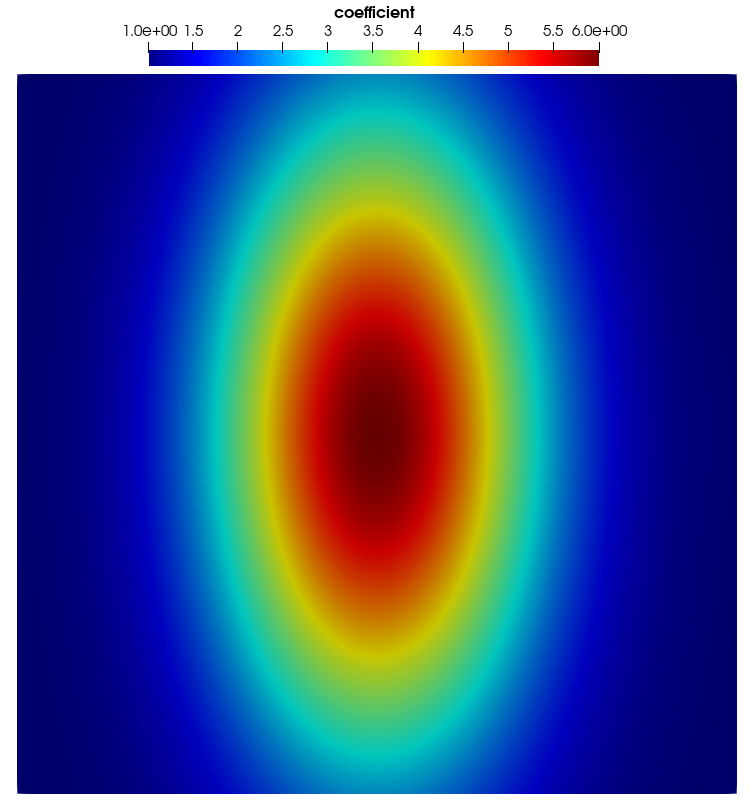}
        \caption{$\kappa$}
        \label{fig:partial_obs_true_coeff}
    \end{subfigure}
    \begin{subfigure}[t]{0.32\textwidth}
        \includegraphics[width=\textwidth]{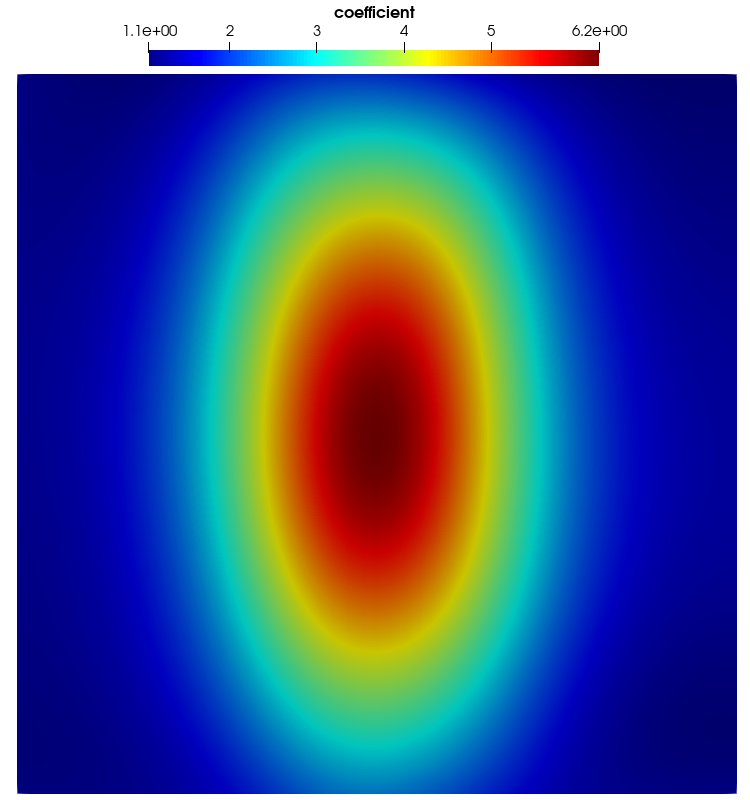}
        \caption{$\kappa^{id}$}
        \label{fig:partial_obs_feinn_coeff}
    \end{subfigure}
    \begin{subfigure}[t]{0.32\textwidth}
        \includegraphics[width=\textwidth]{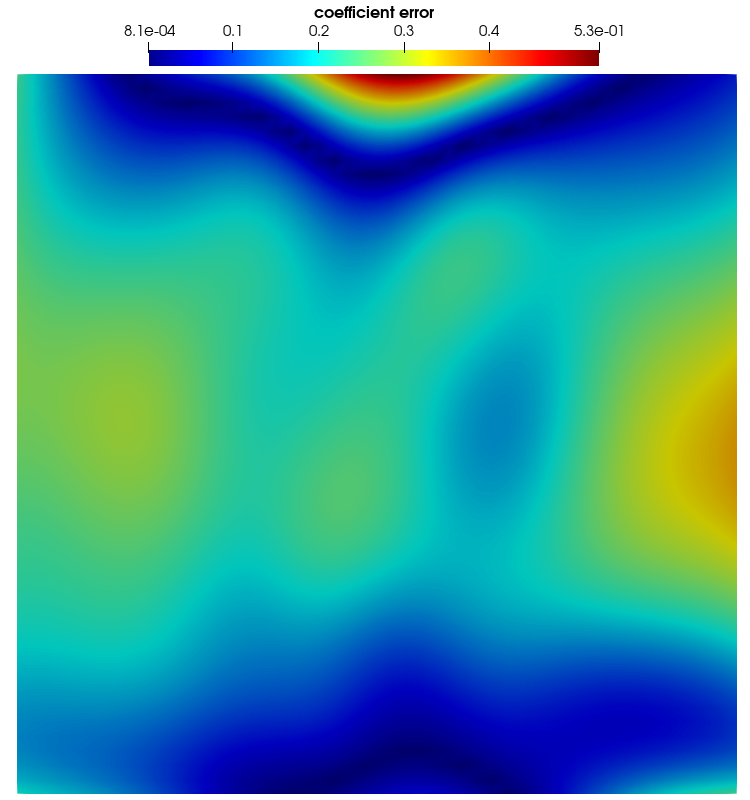}
        \caption{$|\kappa^{id} - \kappa|$}
        \label{fig:partial_obs_coeff_error}
    \end{subfigure}
     
    \caption{Illustration of known analytical solutions (first column), compatible \ac{feinn} solutions (second column), and corresponding point-wise errors (third column) for the inverse Maxwell problem with partial observations. Results correspond to an experiment with a specific \ac{nn} initialisation. The first row depicts the state, while the second row represents the coefficient.}
    \label{fig:inv_partial_obs_results}
\end{figure}

The domain is discretised by $50\times50$ quadrilateral elements. We employ first-order N{\'e}d{\'e}lec elements for $\mathbf{u}_\mathcal{N}$ interpolation, and first-order \ac{cg} elements for $\kappa_\mathcal{N}$ interpolation. The \ac{nn} structures are almost identical: both have 2 hidden layers with 20 neurons each, but the output layer of $\mathbf{u}_\mathcal{N}$ has 2 neurons, while the output layer of $\kappa_\mathcal{N}$ has 1 neuron activated by the rectification function $r(x)$.
For compatible \acp{feinn}, the training iterations for the three steps are $[150, 50, 3\times600]$, totalling 2,000 iterations, and the penalty coefficients are $\alpha=[0.001,0.003,0.009]$. The total training iterations for the adjoint \ac{nn} method is also 2,000.

In the second column of \fig{fig:inv_partial_obs_results}, we present the \ac{feinn} solutions for the state and coefficient, and their corresponding errors are displayed in the third column. The state error (see \fig{fig:partial_obs_state_error}) is reasonably small compared to the true state (see \fig{fig:partial_obs_true_state}), and the identified coefficient (see \fig{fig:partial_obs_feinn_coeff}) is visually very close to the true coefficient (\fig{fig:partial_obs_true_coeff}).

In \fig{fig:inv_partial_obs_training_history}, we compare the relative errors between compatible \ac{feinn} and adjoint \ac{nn} solutions. Note that the $\varepsilon_{L^2}(\kappa^{id})$ for both methods starts with the same value, as the initialisations for $\kappa_\mathcal{N}$ are identical.
Compatible \acp{feinn} converge faster than adjoint \acp{nn} as all three \ac{feinn} error curves fall below the corresponding adjoint \ac{nn} error curves after 400 iterations.
On the downside, the convergence of compatible \acp{feinn} is not as stable as that of adjoint \acp{nn}, as the errors for the former fluctuate during training. This raises concerns about whether the superior performance of compatible \acp{feinn} over adjoint \acp{nn} is due to a specific \ac{nn} initialisation. We investigate this in the following experiments.

\begin{figure}[ht]
    \centering
    \includegraphics[width=\textwidth]{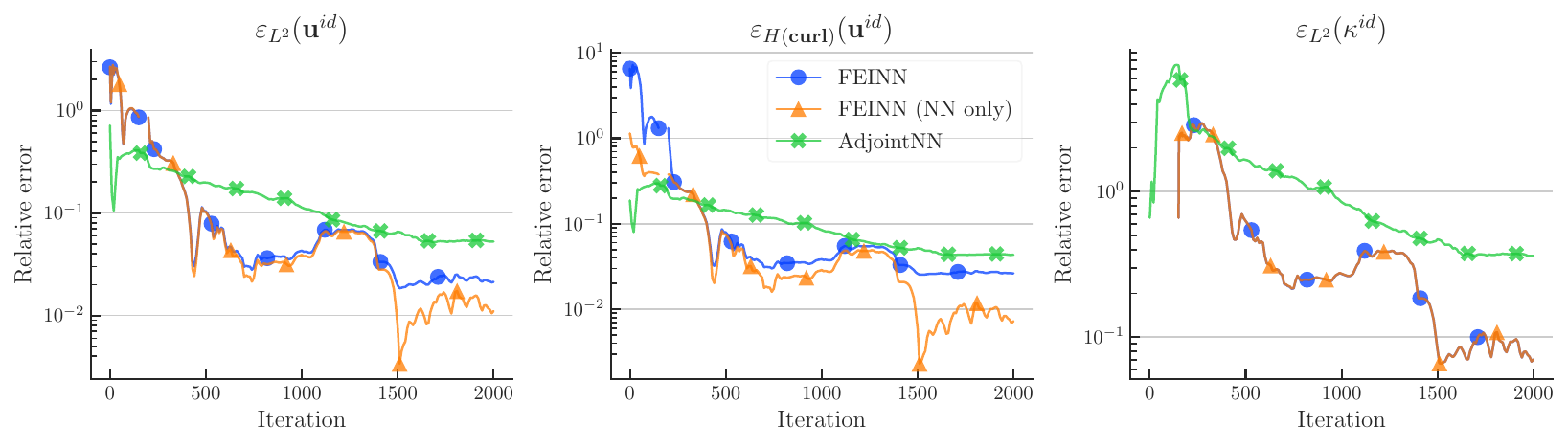}
    \caption{Comparison among \acp{feinn} and adjoint \acp{nn} in terms of relative errors during training for the inverse Maxwell problem with partial observations. Both optimisation loops were run for 2,000 iterations.}
    \label{fig:inv_partial_obs_training_history}
\end{figure}

We repeat the experiment 100 times with different \ac{nn} initialisations for both \acp{feinn} and adjoint \acp{nn}. The resulting box plots for the relative errors are presented in \fig{fig:inv_partial_obs_methods_cmp_boxplot}, where whiskers represent the minimum and maximum values within 1.5 times the interquartile range.
In this scenario with partial observations, the \ac{feinn} method demonstrates greater stability and accuracy compared to the adjoint \ac{nn} method. The \ac{feinn} boxes for all three relative errors are more tightly clustered and positioned lower than the adjoint \ac{nn} boxes.
Looking at the \acp{nn} themselves resulting from the \acp{feinn} method, the state \acp{nn} are more accurate but less stable compared to their interpolation counterparts. This is evident from \fig{fig:inv_partial_obs_methods_cmp_boxplot}, where the \ac{nn} error boxes are lower but more scattered than the interpolated \ac{nn} error boxes.
In line with the findings in~\cite{Badia2024}, without interpolation, the \ac{nn} coefficient errors are roughly the same as their interpolation counterparts. 
Overall, the \ac{feinn} method exhibit higher reliability than the adjoint \ac{nn} method for solving the inverse Maxwell problem with partial observations.

\begin{figure}[ht]
    \centering
    \includegraphics[width=\textwidth]{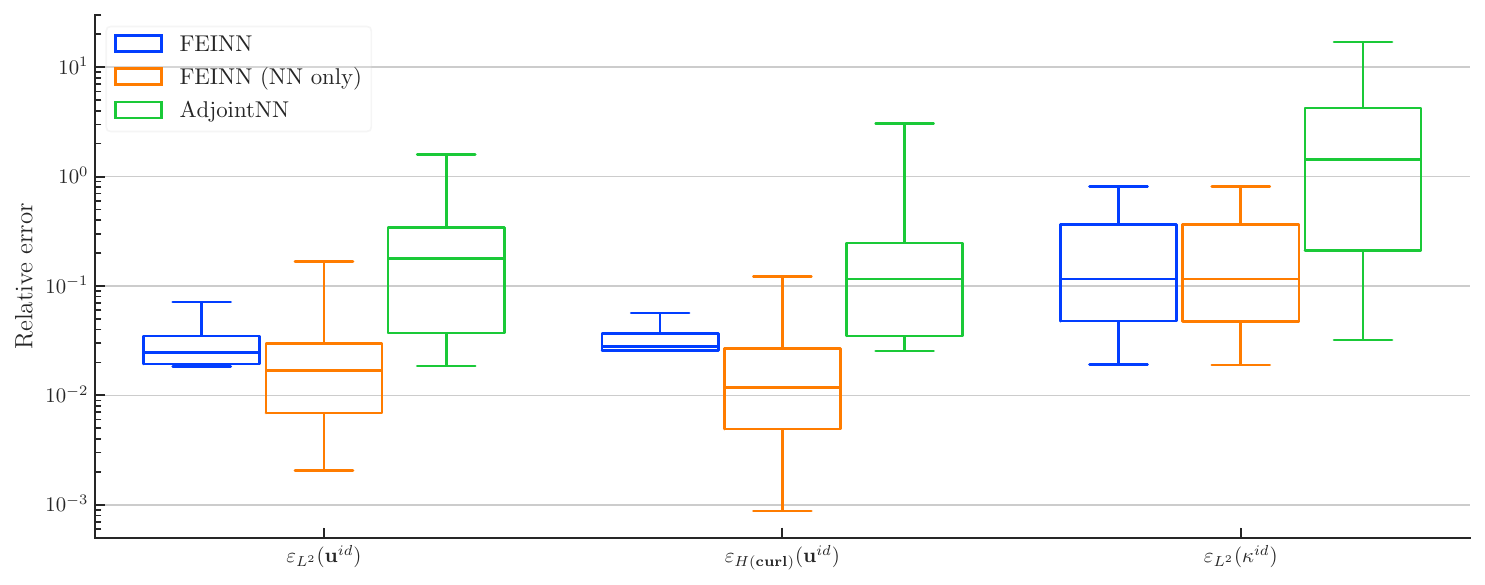}
    \caption{Comparison among \acp{feinn} and adjoint \acp{nn} in terms of relative errors (depicted using box plots) with 100 different \ac{nn} initialisations for the inverse Maxwell problem with partial observations.}
    \label{fig:inv_partial_obs_methods_cmp_boxplot}
\end{figure}

\subsubsection{Noisy observations} \label{subsubsec:inv_noisy_obs}
We consider next an alternative inverse Maxwell problem that now involves noisy observations of the state $\mathbf{u}$. Unlike the previous experiment, we have access to both components of the data, but they are contaminated with Gaussian noise $\epsilon \sim N(0, 0.05^2)$.
Note that each observation is computed by adding Gaussian noise to the true state, i.e., $\mathbf{u}^{\mathrm{obs}} = (1 + \epsilon) \mathbf{u}$. We generate $30^2$ noisy observations distributed uniformly in $[0.005,0.995]^2$.
The domain, boundary conditions, and analytical state are the same as in \sect{subsubsec:inv_partial_obs}, while the analytical coefficient is 
\begin{equation*}
    \kappa(x, y) = \frac{1}{1 + x^2 + y^2 + (x - 1)^2 + (y - 1)^2}.
\end{equation*}
Regarding the \ac{nn} structures, the only change is an increase in the number of neurons in $\kappa_\mathcal{N}$ from 20 to 50. For compatible \acp{feinn}, the training iterations are set to $[150,50,2 \times 400]$, with penalty coefficients $[0.01,0.03]$. The total training iterations for adjoint \acp{nn} are also set to 1,000.

We juxtapose the true solutions, non-interpolated \ac{nn} solutions, and their corresponding errors in \fig{fig:inv_noisy_obs_results}. Note that we use the same analytic state as in the previous experiment; therefore, instead of presenting the state itself, we depict its curl. Additionally, we observe that the \ac{nn} solutions for the state consistently outperform their interpolations. Hence, we only present the non-interpolated \ac{nn} solutions in the figure.
Let us first discuss the state solutions. Since $\mathbf{u}_\mathcal{N} \in \mathcal{C}^\infty(\bar{\Omega})$, $\pmb{\nabla} \times \mathbf{u}_\mathcal{N}$ is also smooth, as confirmed in \fig{fig:inv_noisy_obs_nn_curl_state}. However, upon interpolation onto an edge \ac{fe} space with $\mathcal{C}^0$ normal components, we expect observing discontinuities at the element boundaries in the curl of the interpolated state solution, similar to those displayed in \fig{fig:fwd_darcy_feinn_div_flux}. 
The smoothness of $\mathbf{u}_\mathcal{N}$ enhances the accuracy of the state \ac{nn} solutions in terms of curl error.
In this experiment, the relative curl error of the interpolated $u_\mathcal{N}$ is $2.5\times10^{-2}$, while $u_\mathcal{N}$ itself achieves a lower relative curl error of $1.5\times10^{-3}$.
For the coefficient, the \ac{nn} solution (\fig{fig:inv_noisy_obs_nn_coeff}) successfully identifies the pattern of the true coefficient (\fig{fig:inv_noisy_obs_true_coeff}), with the pointwise error (\fig{fig:inv_noisy_obs_nn_coeff_error}) below 0.05 across most of the domain.

\begin{figure}[ht]
    \centering
    \begin{subfigure}[t]{0.32\textwidth}
        \includegraphics[width=\textwidth]{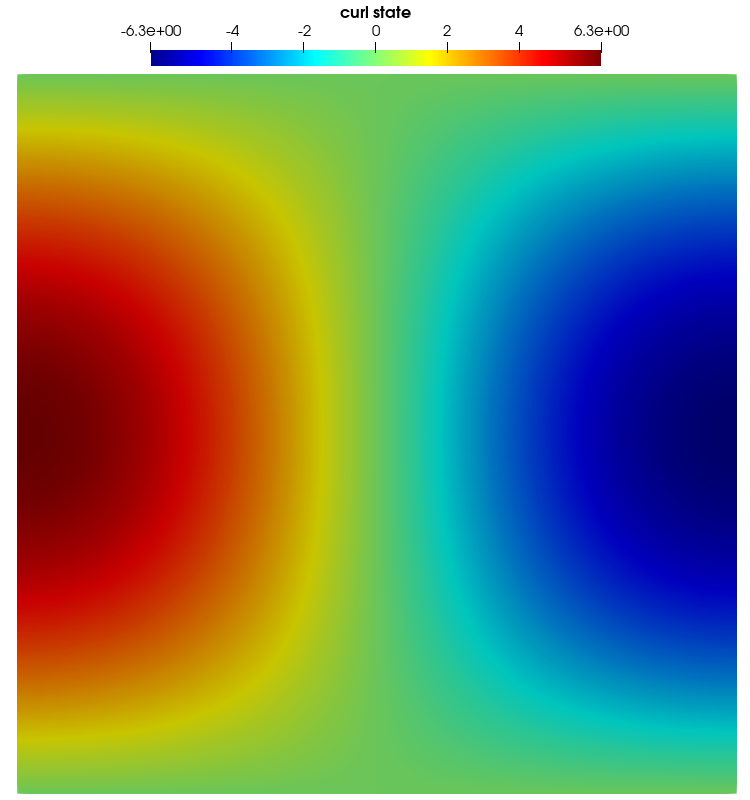}
        \caption{$\pmb{\nabla} \times \mathbf{u}$}
        \label{fig:inv_noisy_obs_true_curl_state}
    \end{subfigure}
    \begin{subfigure}[t]{0.32\textwidth}
        \includegraphics[width=\textwidth]{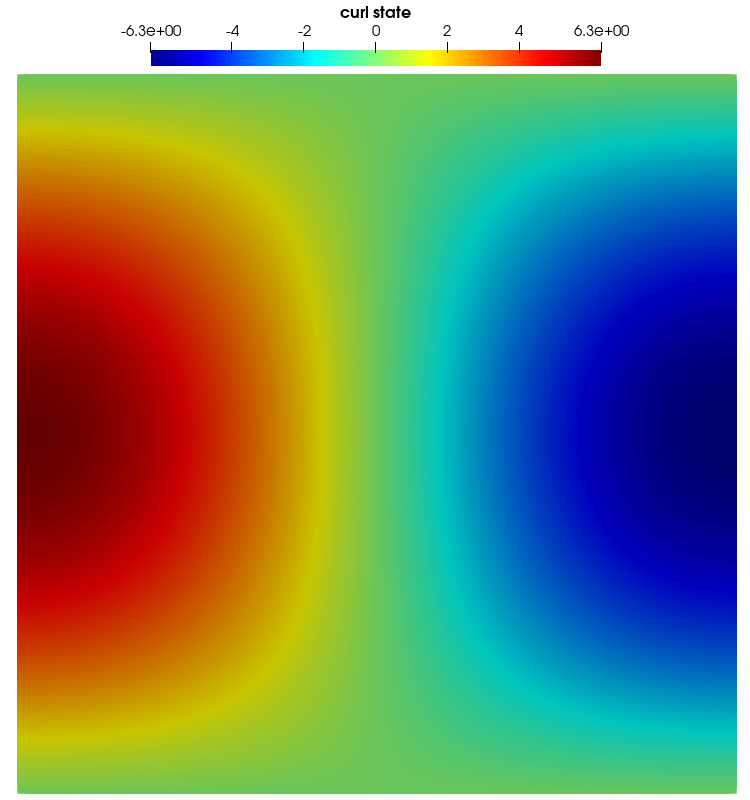}
        \caption{$\pmb{\nabla} \times \mathbf{u}^{id}$}
        \label{fig:inv_noisy_obs_nn_curl_state}
    \end{subfigure}
    \begin{subfigure}[t]{0.32\textwidth}
        \includegraphics[width=\textwidth]{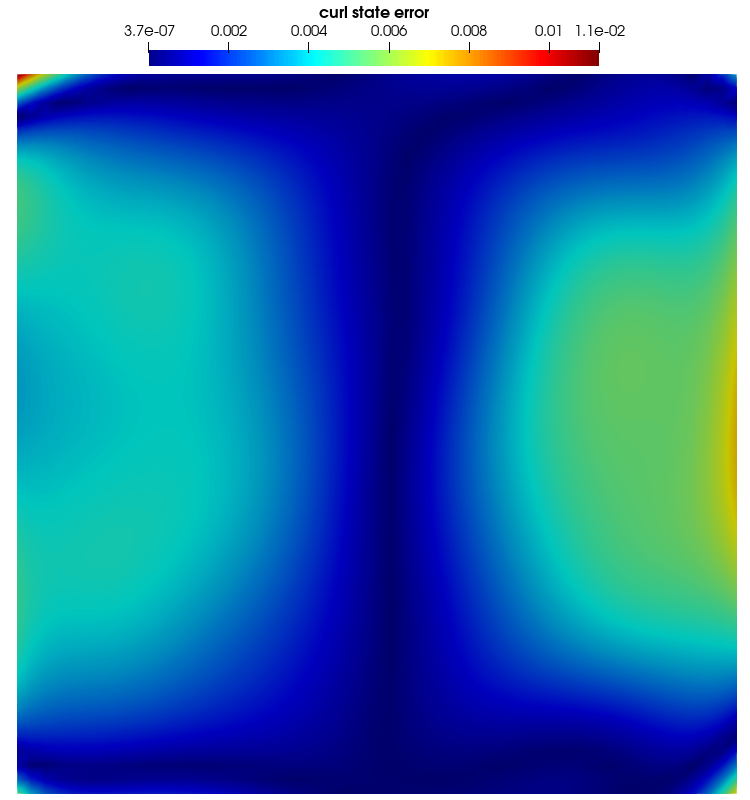}
        \caption{$|\pmb{\nabla} \times (\mathbf{u}^{id} - \mathbf{u})|$}
        \label{fig:inv_noisy_obs_nn_curl_state_error}
    \end{subfigure}

    \centering
    \begin{subfigure}[t]{0.32\textwidth}
        \includegraphics[width=\textwidth]{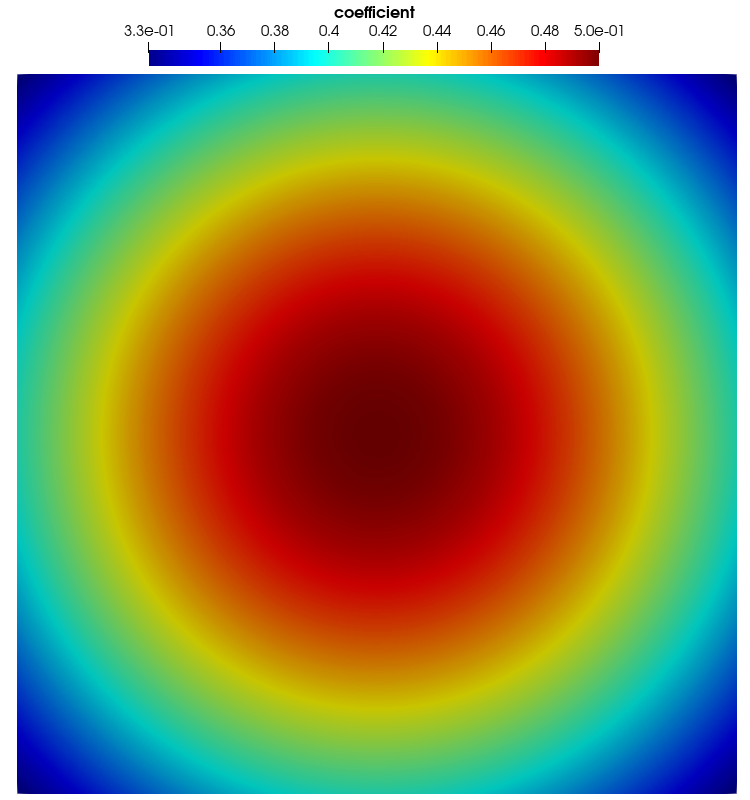}
        \caption{$\kappa$}
        \label{fig:inv_noisy_obs_true_coeff}
    \end{subfigure}
    \begin{subfigure}[t]{0.32\textwidth}
        \includegraphics[width=\textwidth]{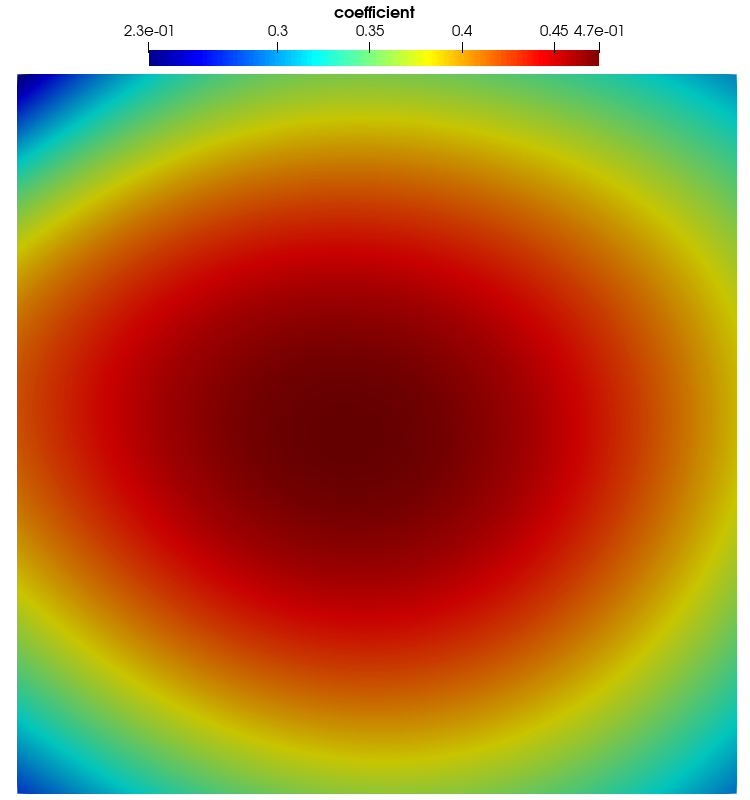}
        \caption{$\kappa^{id}$}
        \label{fig:inv_noisy_obs_nn_coeff}
    \end{subfigure}
    \begin{subfigure}[t]{0.32\textwidth}
        \includegraphics[width=\textwidth]{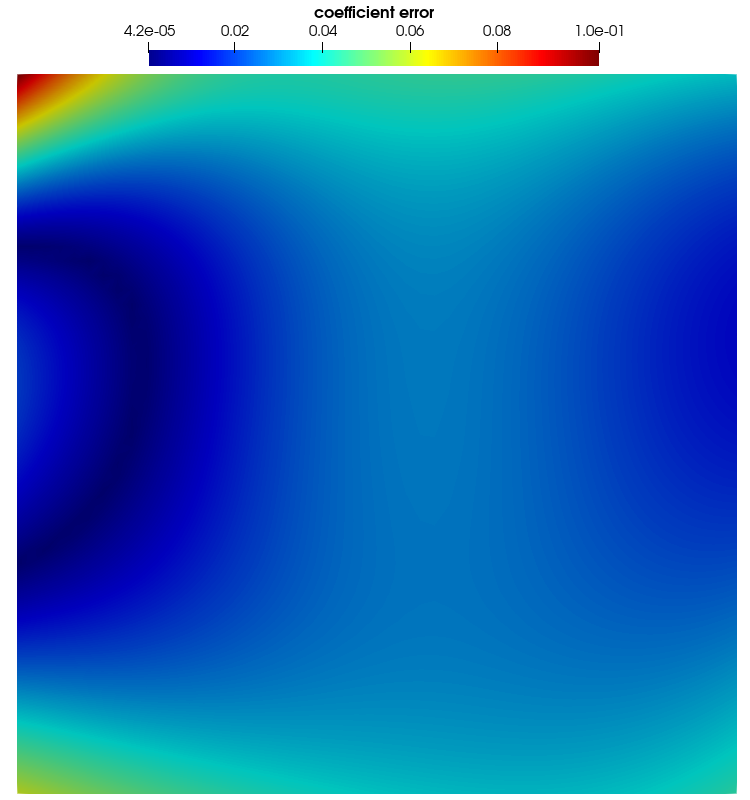}
        \caption{$|\kappa^{id} - \kappa|$}
        \label{fig:inv_noisy_obs_nn_coeff_error}
    \end{subfigure}
     
    \caption{Illustration of known analytical solutions (first column), compatible \ac{feinn} solutions (second column), and corresponding point-wise errors (third column) for the inverse Maxwell problem with noisy observations. Results correspond to an experiment with a specific \ac{nn} initialisation. The first row depicts the curl of the state, while the second row represents the coefficient.}
    \label{fig:inv_noisy_obs_results}
\end{figure}

The training error histories of both \acp{feinn} and adjoint \acp{nn} are presented in \fig{fig:inv_noisy_obs_training_history}. The label tag ``no reg'' for the adjoint \ac{nn} method denotes the history without regularisation. 
It is evident that, although the adjoint \ac{nn} method converges faster than the \ac{feinn} method, it requires explicit regularisation to stabilise the training. We adopt the same $L^1$ regularisation on the parameters of $\kappa_\mathcal{N}$ as in~\cite{Mitusch2021}, with a coefficient $10^{-3}$. 
One advantage of the \ac{feinn} method is that no regularisation is needed to achieve a stable training even though regularisation could help to make the problem better posed and improve convergence.
Moreover, the non-interpolated \acp{nn} have the potential for better state accuracy due to the smoothness of $\mathbf{u}_\mathcal{N}$, as the errors of $\mathbf{u}_\mathcal{N}$ fall below the adjoint \ac{nn} errors after 300 iterations.

\begin{figure}[ht]
    \centering
    \includegraphics[width=\textwidth]{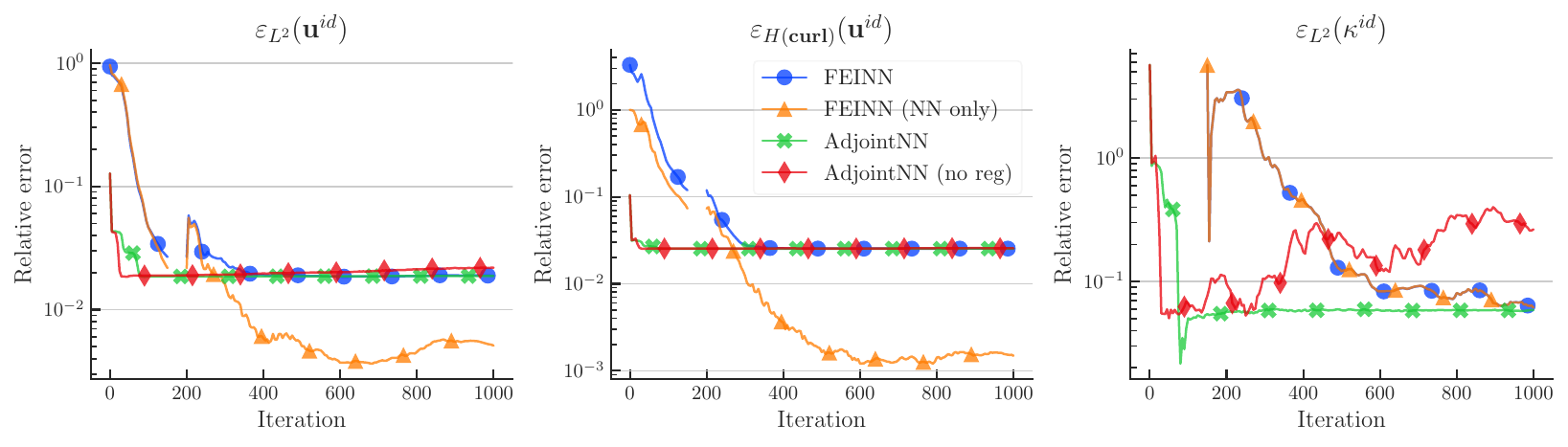}
    \caption{Comparison among \acp{feinn} and adjoint \acp{nn} in terms of relative errors during training for the inverse Maxwell problem with noisy observations. Both optimisation loops were run for 1,000 iterations.}
    \label{fig:inv_noisy_obs_training_history}
\end{figure}

To study the robustness of the methods, we repeat the experiment 100 times with different \ac{nn} initialisations. The resulting box plots for the relative errors are presented on the left side of \fig{fig:inv_noisy_obs_boxplots}. 
For the identified states, both \acp{feinn} and adjoint \acp{nn} exhibit stable and similar performance, as evidenced by the boxes degenerating into lines and being positioned at the same level.
However, the \ac{feinn} method has more potential than the adjoint \ac{nn} method to achieve better accuracy, as the non-interpolated \ac{nn} boxes are noticeably lower than the adjoint \ac{nn} boxes. 
For the coefficient, the adjoint \ac{nn} method is more stable than the \ac{feinn} method, with a more compact box plot. The \ac{feinn} method also yields highly accurate results: the upper quartile (Q3) line of the \ac{feinn} coefficient box is below 9\%, and a similar trend is observed for \acp{nn} without interpolation. 
In summary, the \ac{feinn} method demonstrates comparable performance to the adjoint \ac{nn} method in identifying unknown coefficient, and it demonstrates greater potential in fully recovering the state with noisy data.

\begin{figure}[ht]
    \centering
    \includegraphics[width=\textwidth]{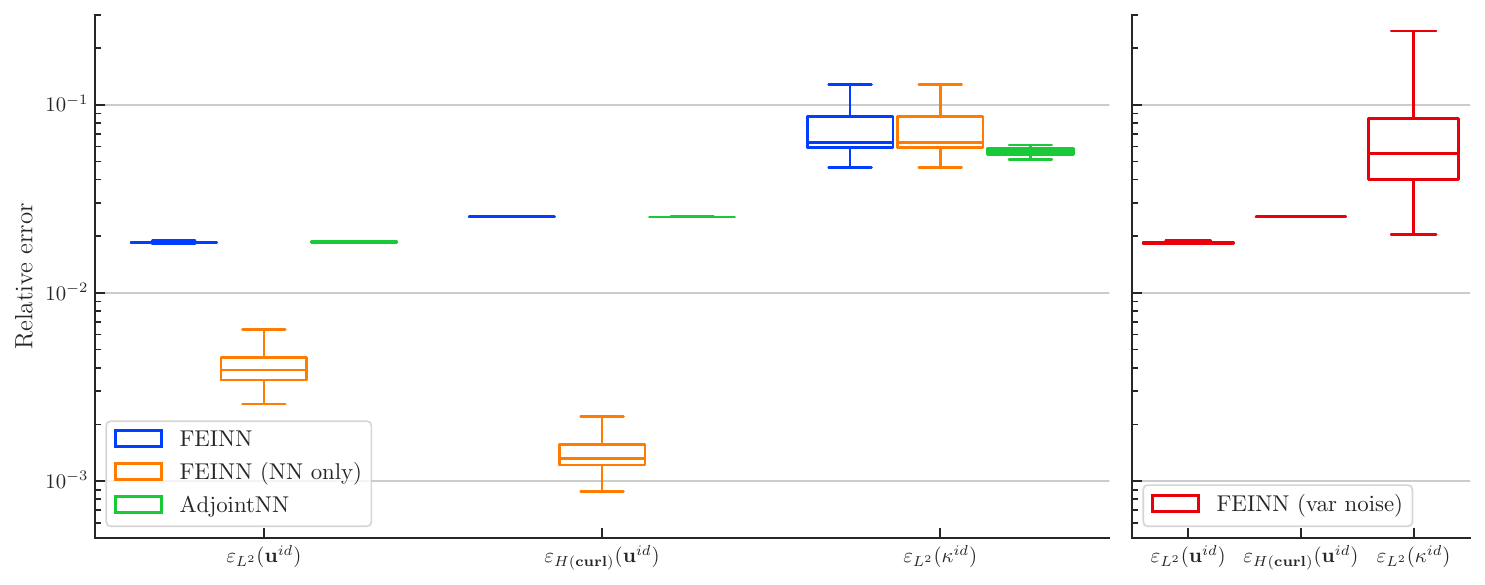}
    \caption{Box plots from 100 experiments for the inverse Maxwell problem with noisy observations. Left: comparison among \acp{feinn} and adjoint \acp{nn} with regularisation in terms of relative errors with different \ac{nn} initialisations. Right: sensitivity analysis of \acp{feinn} to Gaussian noise in observations, generated using different random seeds.}
    \label{fig:inv_noisy_obs_boxplots} 
\end{figure}

In addition to \ac{nn} initialisations, another source of randomness in this problem is the Gaussian noise in the observations. To assess the sensitivity of \acp{feinn} to this noise, we generate 100 different sets of noisy observations using different random seeds and repeat the experiment with the same set of \ac{nn} initialisations.
The resulting box plots for the relative errors are displayed on the right side of \fig{fig:inv_noisy_obs_boxplots}, conforming the robustness of the \ac{feinn} method to random noise. The state boxes are mostly flat, while the Q3 value of the coefficient box remains below 9\%.

\subsubsection{Boundary observations} \label{subsubsec:inv_boundary_obs}
We conclude the inverse experiments section by attacking a boundary value inverse Maxwell problem. This scenario is particularly challenging as state observations are available only around the boundary, while the state in the interior domain is unknown. To add more complexity, we assume that the unknown coefficient $\kappa$ is discontinuous across the domain $\Omega = [0,1]^2$.
The analytical solutions are as follows:
\begin{equation*}
  \mathbf{u}(x, y) = \begin{bmatrix} \cos(\pi x)\cos(y) \\ \sin(x)\sin(\pi y) \end{bmatrix}, \quad \kappa(x, y) = \begin{cases}
    1, & \text{ if } y > 2x\\
    10,& \text{ if } y \leq 2x\\
  \end{cases}.
\end{equation*}
The curl of the analytic state and the discontinuous $\kappa$ are depicted in \fig{fig:inv_boundary_obs_true_curl_state} and \fig{fig:inv_boundary_obs_true_coeff}, respectively. 
A total of $4\times70$ boundary observations are generated by evaluating the true state along the interior side of the boundary. Their locations are marked as purple dots in \fig{fig:inv_boundary_obs_true_curl_state}.

\begin{figure}[ht]
  \centering
  \begin{subfigure}[t]{0.32\textwidth}
    \includegraphics[width=\textwidth]{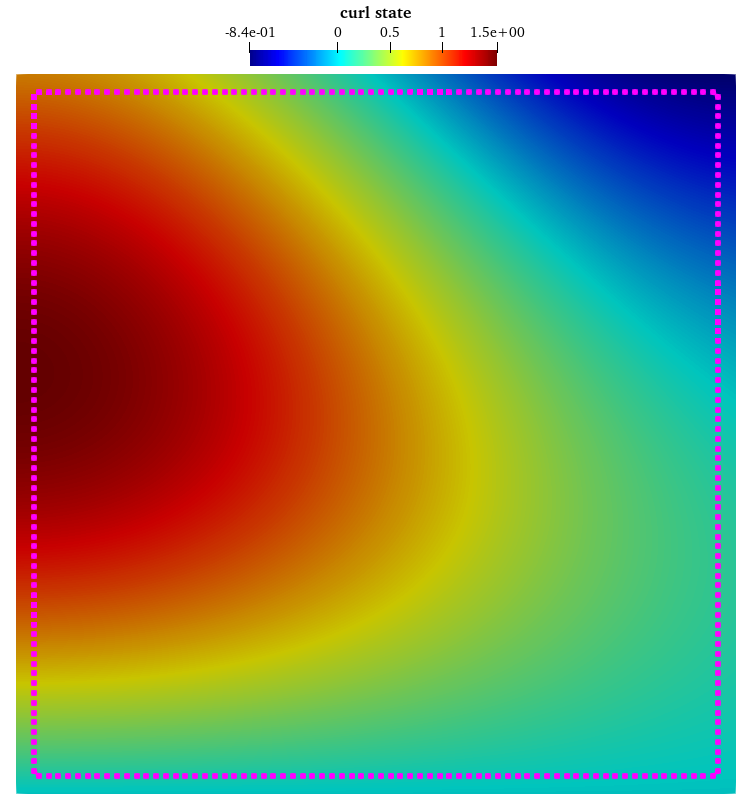}
    \caption{$\pmb{\nabla} \times \mathbf{u}$}
    \label{fig:inv_boundary_obs_true_curl_state}
  \end{subfigure}
  \begin{subfigure}[t]{0.32\textwidth}
    \includegraphics[width=\textwidth]{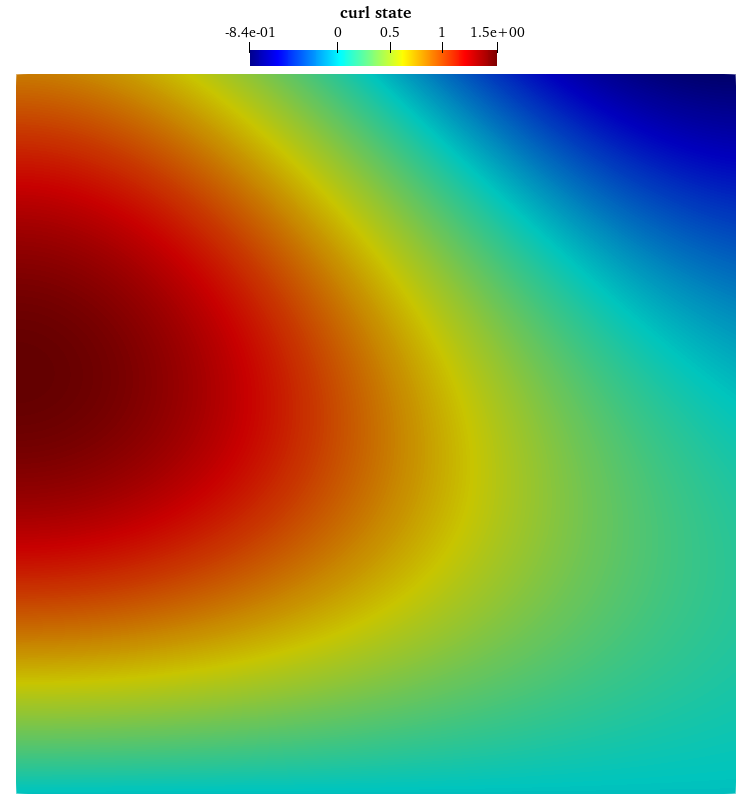}
    \caption{$\pmb{\nabla} \times \mathbf{u}^{id}$}
    \label{fig:inv_boundary_obs_nn_curl_state}
  \end{subfigure}
  \begin{subfigure}[t]{0.32\textwidth}
    \includegraphics[width=\textwidth]{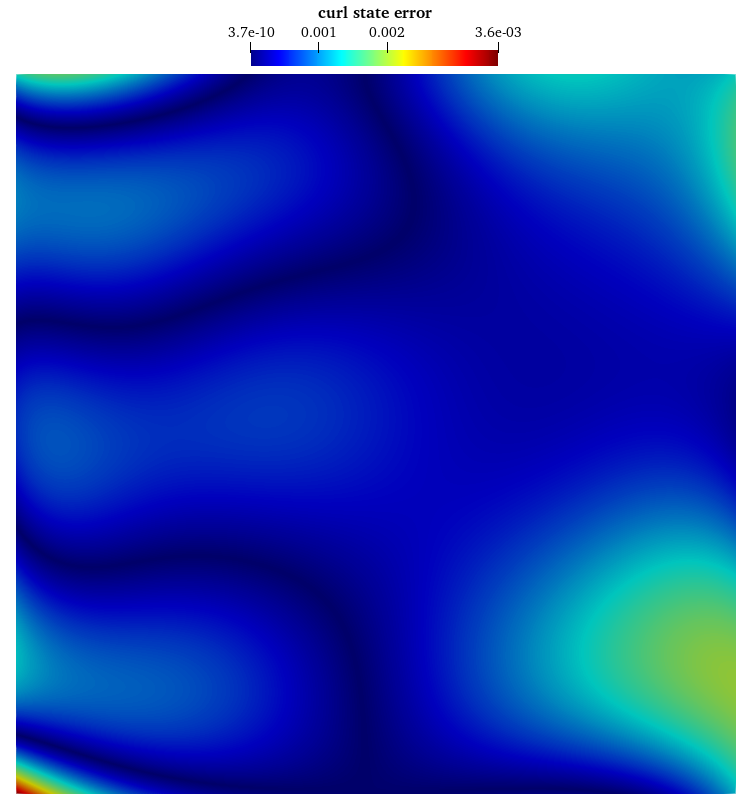}
    \caption{$|\pmb{\nabla} \times (\mathbf{u}^{id} - \mathbf{u})|$}
    \label{fig:inv_boundary_obs_nn_curl_state_error}
  \end{subfigure}

  \centering
  \begin{subfigure}[t]{0.32\textwidth}
      \includegraphics[width=\textwidth]{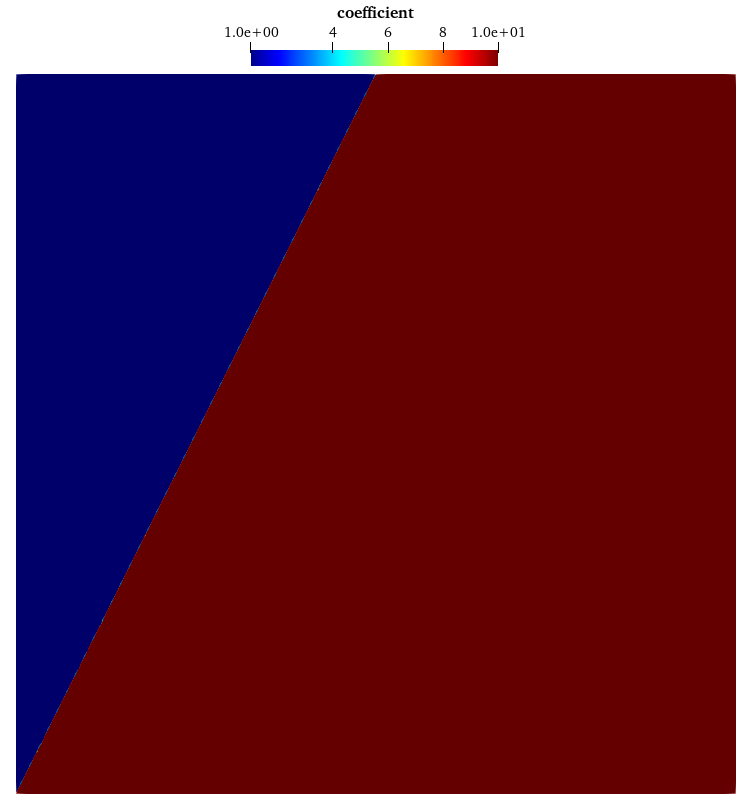}
      \caption{$\kappa$}
      \label{fig:inv_boundary_obs_true_coeff}
  \end{subfigure}
  \begin{subfigure}[t]{0.32\textwidth}
    \includegraphics[width=\textwidth]{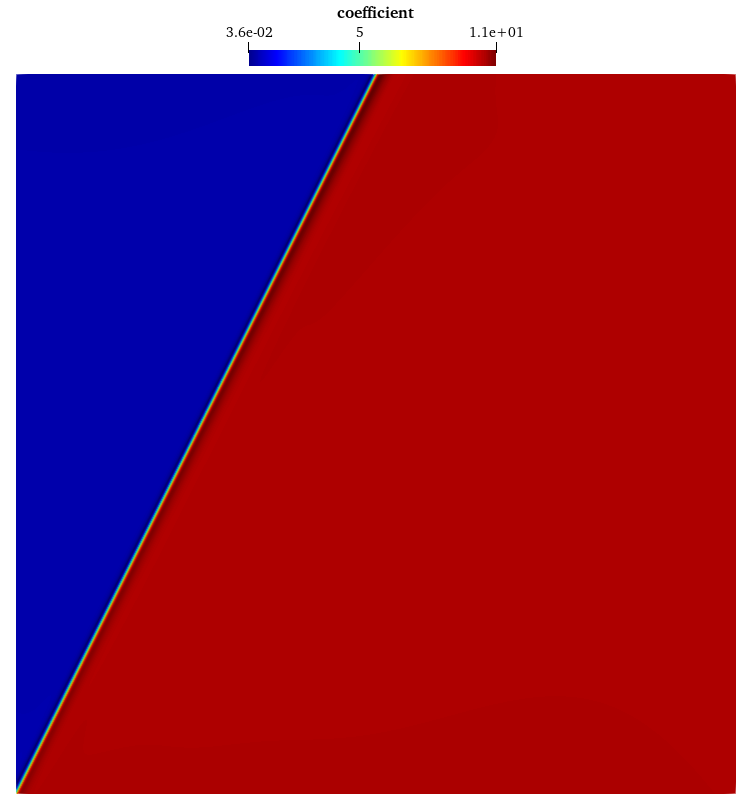}
    \caption{$\kappa^{id}$}
    \label{fig:inv_boundary_obs_nn_coeff}
  \end{subfigure}
  \begin{subfigure}[t]{0.32\textwidth}
    \includegraphics[width=\textwidth]{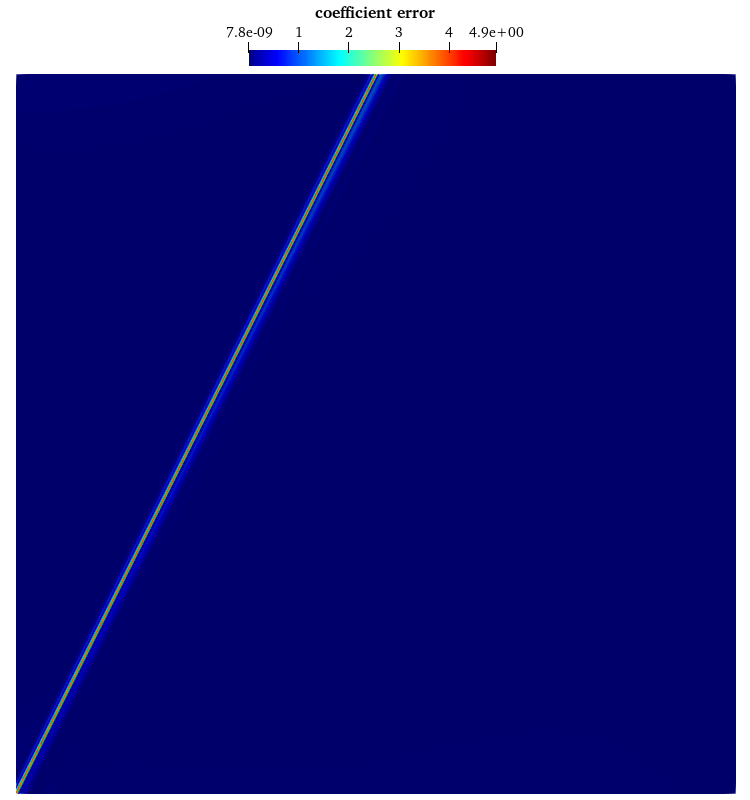}
    \caption{$|\kappa^{id} - \kappa|$}
    \label{fig:inv_boundary_obs_nn_coeff_error}
  \end{subfigure}
    
  \caption{Illustration of analytical and \ac{nn} solutions and the  errors for the inverse Maxwell problem with boundary observations. The purple dots in (a) represent the observed data. See caption of \fig{fig:inv_noisy_obs_results} for details on the information being displayed in this figure.}
  \label{fig:inv_boundary_obs_results}
\end{figure}

In this experiment, we focus exclusively on the compatible \ac{feinn} approach. Since the Dirichlet boundary condition is presented as boundary observations, the interpolation space for $\mathbf{u}_{\mathcal{N}}$ is taken as $U^1_h$ rather than $U^1_{h,0}$ in the previous experiments. To reflect this change in the \ac{fe} space, the loss~\eqref{eq:inverse_loss} is modified as follows:
\begin{equation*}
  \mathscr{L}(\pmb{\kappa}_\mathcal{N}, \mathbf{u}_\mathcal{N}) \doteq \norm{\mathbf{d} - \mathcal{D}_h (\pi_{h}^1(\mathbf{u}_\mathcal{N}))}_{\ell^2} + \alpha \norm{\tilde{\mathcal{R}}_h(\pi^0_h(\pmb{\kappa}_\mathcal{N}), \pi^1_{h}(\mathbf{u}_\mathcal{N}))}_{{V^1}'}, 
\end{equation*}
where $\tilde{\mathcal{R}}(\kappa, \mathbf{u}, \mathbf{v}) = a(\kappa, \mathbf{u}, \mathbf{v}) - \ell(\mathbf{v}) - \int_{\partial \Omega}(\mathbf{v} \times \mathbf{n}) \cdot (\pmb{\nabla} \times \mathbf{u})$ is the modified \ac{pde} residual. The boundary integral term in $\tilde{\mathcal{R}}$ arises from integration by parts and the fact that $\mathbf{v}\times\mathbf{n}$ is not necessarily zero on $\partial \Omega$.

The \ac{nn} structures for $\mathbf{u}_{\mathcal{N}}$ and $\kappa_{\mathcal{N}}$ remain the same as in the noisy observations experiment. The training iterations and penalty coefficients are set to $[150,50,3\times 400]$ and $[0.001,0.003, 0.009]$, respectively. The mesh consist of $100\times100$ squares, and the interpolation space orders are consistent with those in the earlier inverse experiments.
We present the trained \ac{nn} solutions and their errors in the last two columns of \fig{fig:inv_boundary_obs_results}. 
Despite having state observations only around the boundary, the curl of the \ac{nn} state solution produced by our method are visually indistinguishable from the curl of the true state, as evident from the comparison of \fig{fig:inv_boundary_obs_nn_curl_state} and \fig{fig:inv_boundary_obs_true_curl_state}. 
Moreover, the identified \ac{nn} coefficient as shown in \fig{fig:inv_boundary_obs_nn_coeff} successfully captures the jumps of the true coefficient. The pointwise error of the coefficient, depicted in \fig{fig:inv_boundary_obs_nn_coeff_error}, is small in most regions of the domain, with slightly higher errors near the discontinuity.

\fig{fig:inv_boundary_obs_training_history} illustrates the training error histories for the compatible \ac{feinn} method. All relative errors decrease steadily during training and converge to low levels.
Additionally, after a few hundred iterations, $\mathbf{u}_\mathcal{N}$ errors consistently fall below those of its interpolation counterpart in both $L^2$ and $H(\bcurl)$ norms. 
Overall, the training curves confirm the reliability of our 3-step training strategy and the potential for $\mathbf{u}_\mathcal{N}$ to achieve higher accuracy than its \ac{fe} interpolation when trained with boundary observations.
 
\begin{figure}[ht]
  \centering
  \includegraphics[width=\textwidth]{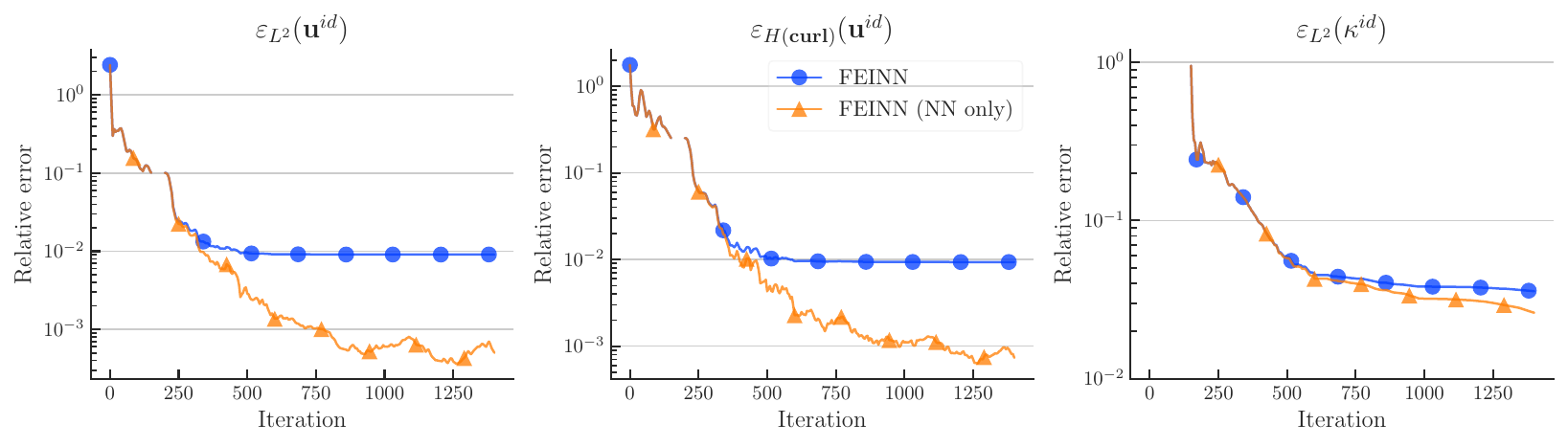}
  \caption{Relative errors during training by the compatible \ac{feinn} method for the inverse Maxwell problem with boundary observations.}
  \label{fig:inv_boundary_obs_training_history}
\end{figure}

Similar to previous experiments, we conduct 100 trials with different \ac{nn} initial parameters to assess our method's robustness in the boundary observations case. The resulting box plots for the relative errors are presented in \fig{fig:inv_boundary_obs_boxplots}. 
For the state, the interpolated \acp{nn} are highly reliable and accurate, with line-like error boxes positioned at 1\% relative errors in both $L^2$ and $H(\bcurl)$ norms. Although non-interpolated \ac{nn} solutions result in slightly more scattered boxes, they demonstrate greater accuracy than their their interpolation counterparts. This is reflected in the upper whisker lines for the $L^2$ and $H(\bcurl)$ relative errors, which are below 0.2\% and 0.3\%, respectively.
Regarding the coefficient, the \acp{nn} achieve a 1\% error reduction compared to their interpolations, with most $\kappa_{\mathcal{N}}$ errors around 3\%. These coefficient error boxes are nearly flat, with an interquartile range of less than 1\%. These results highlight our method's robustness in recovering both the state and discontinuous coefficient of the inverse Maxwell problem using only boundary observations.

\begin{figure}[ht]
  \centering
  \includegraphics[width=0.65\textwidth]{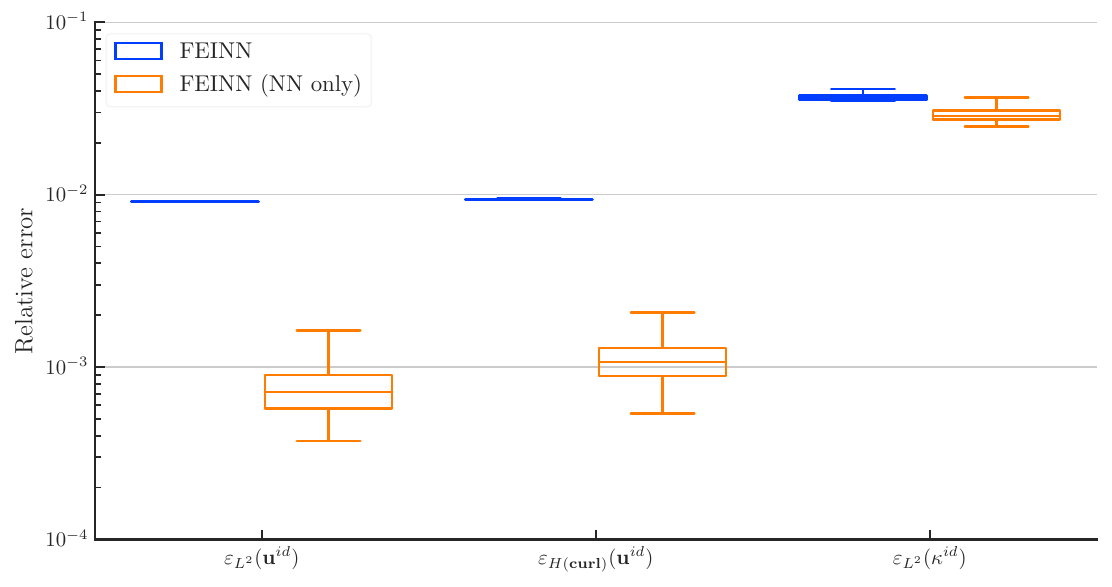}
  \caption{Box plots from 100 experiments for the inverse Maxwell problem with boundary observations.}
  \label{fig:inv_boundary_obs_boxplots} 
\end{figure}

\section{Conclusions} \label{sec:conclusions}
In this work, we extended the \ac{feinn} method~\cite{Badia2024} to solve vector-valued \acp{pde} with weak solutions in $H(\bcurl)$ or $H(\bdiv)$ spaces by interpolating \acp{nn} onto compatible \ac{fe} spaces, i.e., structure-preserving spaces that satisfy a discrete de Rham subcomplex, such as N{\'e}d{\'e}lec and \ac{rt} spaces. We also propose an extension, trace \acp{feinn}, to solve surface Darcy equations by interpolating \acp{nn} onto \ac{fe} spaces defined on the discrete manifold. The compatible \ac{feinn} method retains the advantages of standard \acp{feinn}, such as exact integration, seamless  imposition of strong boundary conditions, and a solid mathematical foundation. Additionally, compatible \acp{feinn} can be easily adapted to inverse problems by adding the data misfit error to the loss function and introducing new \acp{nn} to approximate the unknown physical coefficients.

To evaluate the performance of the methods, we conducted experiments on both forward and inverse problems. Specifically, in the forward Maxwell problem, we analysed how compatible \acp{feinn} perform across different mesh sizes and trial \ac{fe} space orders. To showcase the applicability of compatible \acp{feinn} to surface \acp{pde}, we tackled a forward Darcy problem on the unit sphere.
In general, we observe that using lower-order test bases allows \acp{nn} to beat \ac{fem} solutions in terms of $L^2$ and $H(\bcurl)$ errors for problems with a smooth analytic solution, achieving over two orders of magnitude improvement over \ac{fe} solutions. In addition, interpolated \acp{nn} match the \ac{fe} solutions in all cases. Preconditioning of \acp{nn} (i.e., using dual residual norms in the loss function) improves the stability and efficiency of the training process. Our findings also suggest that the strong \ac{pde} residual with \acp{nn} effectively guides mesh refinement during adaptive training, resulting in \acp{nn} that are more accurate than or comparable to \ac{fem} solutions. For inverse problems, we compare the compatible \ac{feinn} method with the adjoint \ac{nn} method for the inverse Maxwell problem with partial or noisy observations. Compatible \acp{feinn} are more accurate than adjoint \acp{nn} in most cases, do not require explicit regularisation and are robust against random noise in the observations. Besides, the \ac{feinn} method also achieves high accuracy in both state and coefficient for the inverse problem with boundary observations.

Even though \acp{feinn} have been shown to be more accurate and stable than other \ac{pinn}-like methods~\cite{Badia2024,Badia2024adaptive,}, the cost of the training is still high. For forward problems, despite the gains in accuracy compared to \ac{fem} on the same mesh, computational cost is higher than \acp{fem}  with similar accuracy in general. Adaptive and inverse problems improve the efficiency of the method compared to standard approaches, since nested loops are not required and \ac{nn} built-in regularisation can be exploited. We also expect these schemes to be more effective for transient simulations, where the \ac{nn} can be initialised with the previous time step training. In the future, the design of novel nonconvex optimisation strategies for \ac{nn} approximation of \acp{pde}, e.g., using multigrid and domain decomposition-like techniques, will be critical to reduce the computational cost of training of \ac{pinn}-like strategies and efficiently exploit their nonlinear approximability properties.

\section*{Acknowledgments}
This research was partially funded by the Australian Government through the Australian Research Council (project numbers DP210103092 and DP220103160). This work was also supported by computational resources provided by the Australian Government through NCI under the NCMAS and ANU Merit Allocation Schemes. 
W. Li gratefully acknowledges the Monash Graduate Scholarship from Monash University, the NCI computing resources provided by Monash eResearch through Monash NCI scheme for HPC services, and the support from the Laboratory for Turbulence Research in Aerospace and Combustion (LTRAC) at Monash University through the use of their HPC Clusters.

\section*{Declaration of generative AI and AI-assisted technologies in the writing process}
During the preparation of this work the authors used ChatGPT in order to improve language and readability. After using this tool/service, the authors reviewed and edited the content as needed and take full responsibility for the content of the publication.

\printbibliography

\end{document}